\documentclass[10pt]{amsart}
\usepackage{amsmath,amssymb,amsthm,enumerate,theoremref}
\usepackage{mathrsfs} 
\title[Chain conditions and descriptive set theory]{Chain conditions, elementary amenable groups, and descriptive set theory}

\author{Phillip Wesolek}
\address{  Universit\'{e} catholique de Louvain,
   Institut de Recherche en Math\'{e}matiques et Physique (IRMP),
   Chemin du Cyclotron 2, box L7.01.02,
   1348 Louvain-la-Neuve, Belgium}
\email{phillip.wesolek@uclouvain.be}

\author{Jay Williams}
\address{ Department of Mathematics,
	California Institute of Technology,
	Pasadena, CA 91125}
\email{jaywill@caltech.edu}

\date{February 2015}

\newtheorem{thm}{Theorem}[section] 

\newtheorem{obs}[thm]{Observation}
\newtheorem{prop}[thm]{Proposition}
\newtheorem{lem}[thm]{Lemma}
\newtheorem{cor}[thm]{Corollary}

\theoremstyle{definition}
\newtheorem{defn}[thm]{Definition}
\newtheorem{quest}[thm]{Question}
\newtheorem{rmk}[thm]{Remark}

\newtheorem*{claim*}{Claim}

\newcommand{\Zb}{\mathbb{Z}}
\newcommand{\Nb}{\mathbb{N}}

\newcommand{\mc}[1]{\mathcal{#1}}

\newcommand{\ms}[1]{\mathscr{#1}} 

\newcommand{\bfs}[1]{\boldsymbol{#1}}

\newcommand{\acts}{\curvearrowright}

\newcommand{\wbaire}{\mathbb{N}^{<\mathbb{N}}}

\newcommand{\conc}{^{\smallfrown}}

\newcommand{\Fw}{\mathbb{F}_{\omega}}

\newcommand{\Gw}{\mathscr{G}}

\newcommand{\from}{\colon}
\renewcommand{\to}{\rightarrow}
\newcommand{\surject}{\twoheadrightarrow}

\newcommand{\<}{\langle}
\renewcommand{\>}{\rangle}
\newcommand{\ngrp}[1]{\langle\langle #1 \rangle\rangle}

\newcommand{\normal}{\trianglelefteq}

\newcommand{\grp}[1]{\langle #1 \rangle}
\newcommand{\ol}[1]{\overline{#1}}

\newcommand{\rk}{\mathop{\rm rk}}
\newcommand{\EA}{\mathop{\rm EG} \nolimits}
\newcommand{\AG}{\mathop{\rm AG} \nolimits}

\begin{document}

\begin{abstract}
We first consider three well-known chain conditions in the space of marked groups: the minimal condition on centralizers, the maximal condition on subgroups, and the maximal condition on normal subgroups. For each condition, we produce a characterization in terms of well-founded descriptive-set-theoretic trees. Using these characterizations, we demonstrate that the sets given by these conditions are co-analytic and not Borel in the space of marked groups. We then adapt our techniques to show elementary amenable marked groups may be characterized by well-founded descriptive-set-theoretic trees, and therefore, elementary amenability is equivalent to a chain condition. Our characterization again implies the set of elementary amenable groups is co-analytic and non-Borel. As corollary, we obtain a new, non-constructive, proof of the existence of finitely generated amenable groups that are not elementary amenable.
\end{abstract}

\maketitle

\section{Introduction}

Chain conditions appear frequently in the study of countable groups. These are finiteness conditions that forbid certain infinite sequences of subgroups. An elementary but interesting example of such a condition is the property of being polycyclic. From a geometric group theory perspective, these finiteness conditions ought to restrict the complexity of the groups, as in the case of polycyclic groups. From a descriptive set theory perspective, however, the chain conditions are non-Borel co-analytic statements and, therefore, either admit ``nice" non-chain-condition characterizations - e.g.\ polycyclic groups are soluble with each term of the derived series finitely generated - or describe large and wild classes. In this work, we explore this tension in four chain conditions in the space of marked groups.\par

\indent In the space of marked groups, denoted $\Gw$, we first consider three well-known chain conditions: the minimal condition on centralizers, the maximal condition on subgroups, and the maximal condition on normal subgroups. We characterize each of these in terms of well-founded descriptive-set-theoretic trees. This characterization implies the classes in question are large and wild, whereby they do not admit ``nice" characterizations.

\begin{thm}
Each of the subsets of $\Gw$ defined by the minimal condition on centralizers, the maximal condition on subgroups, and the maximal condition on normal subgroups are co-analytic and not Borel. This remains true when restricting to finitely generated groups.
\end{thm}

Our techniques additionally give new ordinal-valued isomorphism invariants unbounded below the first uncountable ordinal in the cases of the minimal condition on centralizers and the maximal condition on subgroups. The ordinal-valued isomorphism invariant we obtain in the case of the maximal condition on normal subgroups is not new and has been considered in the literature; cf. \cite{C11}. However, our approach is new, and we show that this invariant is unbounded below the first uncountable ordinal.\par

\indent We next consider the set of elementary amenable marked groups. We likewise characterize these in terms of descriptive-set-theoretic trees. It follows that elementary amenability is indeed a chain condition.

\begin{thm}
A countable group $G$ is elementary amenable if and only if there is no infinite descending sequence of the form 
$$G=G_0\geq G_1\geq\ldots\geq G_n \geq \ldots $$
such that for all $n\geq 0$, $G_n\neq\{e\}$ and there is a finitely generated subgroup $K_n\leq G_n$ with $G_{n+1}= [K_n,K_n]\cap H_n$, where $H_n$ is the intersection of the index-$(\leq(n+1))$ normal subgroups of $K_n$.
\end{thm}

Our characterization gives two new invariants of elementary amenable groups: the decomposition rank and decomposition degree. We further obtain
\begin{thm} 
The sets of elementary amenable groups and finitely generated elementary amenable groups are co-analytic and non-Borel in the space of marked groups.
\end{thm}
It is well-known that the set of amenable groups is Borel in the space of marked groups.  Our theorem thus gives a non-constructive answer to an old question of M. Day \cite{D57}, which was open until R. I. Grigorchuk \cite{G84} constructed groups of intermediate growth: \textit{Are all finitely generated amenable groups elementary amenable?} 
\begin{cor}
There is a finitely generated amenable group that is not elementary amenable.
\end{cor}

The paper is organized as follows.  In Section \ref{sec:Prelim}, we discuss the basic properties of $\Gw$ and introduce concepts from descriptive set theory.  In Sections \ref{sec:MinCent},\ref{sec:Max}, and \ref{sec:MaxN}, we analyze sets of groups satisfying various chain conditions.  This introduces our use of descriptive-set-theoretic trees to study the structure of groups as well as the ordinal-valued invariants arising from those trees.  In Section \ref{sec:EAGroups}, we use those same techniques to analyze elementary amenable groups. In Section \ref{sec:Borel}, we prove the maps used throughout the paper are indeed Borel.  Those who are content to believe that our constructions are Borel can safely skip this section without missing any group-theoretic content.  Finally, Section \ref{sec:Remarks} discusses some questions arising from this paper not touched upon in earlier sections.

\section{Preliminaries}\label{sec:Prelim}

\subsection{The space of marked groups}
In order to apply the techniques of descriptive set theory to groups, we need an appropriate space of groups.  Let $\Fw$ be the free group on the letters $\{a_i\}_{i\in\Nb}$; so $\Fw$ is a free group on countably many generators with a distinguished set of generators. The power set of $\Fw$ may be naturally identified with the Cantor space $\{0,1\}^{\Fw}=:2^{\Fw}$. It is easy to check the collection of normal subgroups of $\Fw$, denoted $\Gw$, is a closed subset of $2^{\Fw}$ and, hence, a compact Polish space.  Each $N\in\Gw$ is identified with a \textbf{marked group}. That is the group $G=\Fw/N$ along with a distinguished generating set $\{f_N(a_i)\}_{i\in \Nb}$ where $f_N:\Fw\rightarrow G$ is the usual projection; we always denote this projection by $f_N$. For a marked group $G$, we abuse notation and say $G\in \Gw$; of course, we formally mean $G=\Fw/N$ for some $N\in \Gw$. Since every countable group is a quotient of $\Fw$, $\Gw$ gives a compact Polish space of all countable groups. A sub-basis for this topology is given by sets of the form
$$ O_{\gamma} := \left\{N\in\Gw \mid \gamma\in N\right\}, $$
where $\gamma\in\Fw$ along with their complements.

\indent Similar reasoning leads us to define the space of \textbf{$m$-generated marked groups} as
\[
\Gw_m := \bigcap_{i\geq m}\{N \normal \Fw \mid a_i\in N\}.
\]
This is a closed subset of $\Gw$ and so is a compact Polish space in its own right. We further let $\Gw_{fg}:=\cup_{m\geq 1} \Gw_m$ be the space of finitely generated marked groups.  As this is an $F_\sigma$ subset of $\Gw$, it is a standard Borel space, with Borel sets precisely those sets of the form $\Gw_{fg}\cap B$ with $B$ Borel in $\Gw$; a \textbf{standard Borel space} is a Borel space which admits a Polish topology that induces the Borel structure.  We can thus also talk about Borel functions with domain $\Gw_{fg}$.\par

\indent It is convenient to give the marked groups $G=\Fw/N$ a preferred enumeration. To this end, we fix an enumeration $\bfs{\gamma}:=(\gamma_i)_{i\in \Nb}$ of $\Fw$. Each $G$ is thus taken to come with an enumeration $f_N(\bfs{\gamma}):=(f_N(\gamma_i))_{i\in \Nb}$; note the enumeration of $G$ may have many repetitions.  When we write $G$ as $G=\{g_0,g_1,\ldots\}$, we will always mean this enumeration.  Later in the paper we will work with $\wbaire$, i.e. the set of finite sequences of natural numbers.  If $(s_0,\ldots,s_n)=:s\in\wbaire$, we will write $\{g_s\}$ for the set $\{g_{s_0},\ldots,g_{s_n}\}$.  Note that this set may have fewer than $n+1$ elements, e.g.\ if $s_0=s_1=\ldots=s_n$, or even if the $s_i$ are distinct but enumerate the same element.\par

\indent We will often discuss quotients of groups or particular subgroups of groups, and of course we wish to view these as elements of $\Gw$.  A quotient of a marked group is obviously again a marked group. However, subgroups of marked groups do not have an obvious marking.  The enumeration gives us a preferred way to select markings for subgroups.  If $H\leq \Fw/N=G\in \Gw$, let $\pi_H\from\Fw\to\Fw$ be induced by mapping the generators $(a_i)_{i\in \Nb}$ of $\Fw$ as follows:
\[
\pi_H(a_j):=
\begin{cases}
\gamma_j, & \text{ if }f_N(\gamma_j)\in H \\
e, & \text{ else.}
\end{cases}
\]
We then identify $H$ with $\Fw/\ker (f_N\circ\pi_H)$. In the case $H$ has a distinguished finite generating set $\{g_{i_0},\dots,g_{i_n}\}$, we instead define $\pi_H(a_{i_j})=\gamma_{i_j}$ and $\pi_H(a_j)=e$ for $j\neq i_k$; this streamlines our proofs later. We often appeal to this convention implicitly.\par 

\indent We will consider maps from and on $\Gw$. A slogan from descriptive set theory is ``Borel = explicit'' meaning if you describe a map ``explicitly'', i.e. without an appeal to something like the axiom of choice, it should be Borel.  All of the maps we discuss in the next few sections will be ``explicit'' in this sense, so we will not prove they are Borel when we define them, in order to keep the focus on the group-theoretic aspects of our constructions.  We will often use enumerations of groups in our constructions, but this will not require choice since every marked group comes with a preferred enumeration.  For those who are interested in the details, we discuss the descriptive-set-theoretic aspects of our constructions in Section \ref{sec:Borel}.

\subsection{Descriptive set theory}
\indent We are interested in certain types of non-Borel subsets of $\Gw$.  The following definitions and theorems are all fundamental in descriptive set theory; a standard reference is \cite{K95}.

\begin{defn}
Let $X,Y$ be uncountable standard Borel spaces.  Then $A\subseteq Y$ is \textbf{analytic} (denoted $\Sigma^1_1$) if there is a Borel set $B \subseteq X\times Y$ such that $\operatorname{proj}_Y(B)=A$. A set $C\subseteq Y$ is \textbf{co-analytic} (denoted $\Pi^1_1$) if $Y\setminus C$ is analytic. 
\end{defn}

Every Borel set is analytic, but any uncountable standard Borel space contains non-Borel analytic sets.  It follows that there are non-Borel co-analytic sets.  The collection of analytic sets is closed under countable unions, countable intersections, and Borel preimages.  It follows the collection of co-analytic sets is closed under countable unions, countable intersections, and Borel preimages.  We remark that sets defined using a single existential quantifier which ranges over an uncountable standard Borel space are often analytic as such quantification can typically be rewritten as a projection of a Borel set.  Thus sets defined by using a universal quantifier over an uncountable set are often co-analytic.

\begin{defn}
Let $X,Y$ be standard Borel spaces, and $A\subseteq X$, $B\subseteq Y$.  We say that $A$ \textbf{Borel reduces} to $B$ if there is a Borel map $f\from X\to Y$ such that $f^{-1}(B)=A$.
\end{defn}

If $A$ Borel reduces to $B$ and $B$ is Borel, analytic, or co-analytic, then so is $A$.  This gives us a method for proving that sets are, for example, co-analytic simply by showing they Borel reduce to a co-analytic set.  One important example comes from the space of (descriptive-set-theoretic) trees.

\begin{defn}
A set $T\subseteq\wbaire$ of finite sequences of natural numbers is a \textbf{tree} if it is closed under initial segments.  A sequence $x\in\Nb^\Nb$ is a branch of $T$ if for all $n\in\Nb$, $x\restriction n\in T$. For $s\in T$, $T_s:=\{r\in \wbaire\mid s\conc r\in T\}$ where ``$\conc$" indicates concatenation of finite sequences.
\end{defn}

As with groups, we may identify $X\subseteq\wbaire$ with an element $f_X\in 2^{\wbaire}$.  We define
\[
Tr := \{ x\in 2^{\wbaire} \mid x \text{ is a tree }\}. 
\]
The set $Tr$ is a closed subset of $2^{\wbaire}$ and so is a compact Polish space.  A sub-basis for the topology on $Tr$ is given by sets of the form
$$ O_t := \left\{T\in Tr \mid t\in T\right\}, $$
where $t\in\wbaire$ along with their complements.

\indent There are two subsets of $Tr$ of particular interest to us: 
\[
IF := \{ T\in Tr \mid T \text{ has a branch } \}
\] 
and $WF := Tr\setminus IF$. We call $WF$ the set of \textbf{well-founded} trees and $IF$ the set of \textbf{ill-founded} trees. One can check that $IF$ is analytic, so $WF$ is co-analytic.  The importance of these sets comes from the following fact.

\begin{thm}\cite[Theorem 27.1]{K95}\label{thm:WFComplete}
Every analytic set Borel reduces to $IF$. Therefore, every co-analytic set Borel reduces to $WF$.
\end{thm}
\noindent Thus a set $A$ is co-analytic if and only if it Borel reduces to $WF$.\par

\indent We are interested in $WF$ for a second reason.  Let $ORD$ denote the class of ordinals.  For any $T\in WF$, we can define a function $\rho_T\from T\to ORD$ inductively as follows: If $t\in T$ has no extensions in $T$, let $\rho_T(t)=0$.  Otherwise let $\rho_T(t) = \sup\{\rho_T(s)+1 \mid t\subsetneq s \}$. We may then define a rank function $\rho\from Tr\to ORD$ by 
\[
\rho(T) =
\begin{cases}
\rho_T(\emptyset)+1, & \text{if }T\in WF\\
\omega_1, & \text{else.}
\end{cases}
\]
For $T=\emptyset$, we define $\rho(T)=0$. The function $\rho$ is bounded above by $\omega_1$, the first uncountable ordinal. Furthermore, this rank function has a special property:

\begin{defn}\label{def:PiRank}
Let $X$ be a standard Borel space and $A\subseteq X$.  A function $\phi\from A\to ORD$ is a \textbf{$\Pi^1_1$-rank} if there are relations $\leq_\phi^\Pi$, $\leq_\phi^\Sigma\subseteq X\times X$ such that $\leq_\phi^\Pi$ is co-analytic, $\leq_\phi^\Sigma$ is analytic, and for all $y\in A$,
\begin{align*}
x\in A \wedge \phi(x)\leq \phi(y) &\Leftrightarrow x \leq_\phi^\Sigma y \\
 &\Leftrightarrow x \leq_\phi^\Pi y.
\end{align*}
\end{defn}

Given any rank function on $A$, one may use it to define an order $\leq_\phi$ on $A$.  The idea of the above definition is that if $\phi$ is a $\Pi^1_1$-rank, then the initial segments of $\leq_\phi$ are Borel, and this is witnessed in a uniform way.

\begin{thm}\cite[Exercise 34.6]{K95}
The function $\rho\from WF\to ORD$ is a $\Pi_1^1$-rank.
\end{thm}
\noindent We may use this fact to create other $\Pi_1^1$-ranks in an easy way:  Let $X$ be a standard Borel space.  If $A\subseteq X$ Borel reduces to $WF$ via $f$, then the map $x \mapsto \rho(f(x))$ is a $\Pi_1^1$-rank. \par

\indent The most important fact about $\Pi_1^1$-ranks for this paper is the following (\cite[Theorem 35.23]{K95}):

\begin{thm}[The Boundedness Theorem for $\Pi_1^1$-ranks]\label{thm:BddnessThm}
Let $X$ be a standard Borel space, $A\subseteq X$ co-analytic, and $\phi\from A\to\omega_1$ a $\Pi_1^1$-rank.  Then
$$A \text{ is Borel} \;\Longleftrightarrow\; \sup \{ \phi(x) \mid x\in A \} < \omega_1.$$
\end{thm}

We will use the Boundedness Theorem to show that certain $\Pi_1^1$ sets are not Borel, by showing that they come with $\Pi^1_1$-ranks with images unbounded below $\omega_1$.  To this end, we will often use the following fact about the ranks of trees, which follows immediately from the definition.

\begin{lem}\label{lem:TrRkMonotone}
Suppose $S,T$ are trees and $\phi\from S\to T$ is a map such that $s\subsetneq t \Rightarrow \phi(s) \subsetneq \phi(t)$.  (We call such a map \textbf{monotone}.)  Then $\rho_S(s)\leq \rho_T(\phi(s))$ for all $s\in S$.  In particular $\rho(S)\leq\rho(T)$.
\end{lem}

\section{The minimal condition on centralizers}\label{sec:MinCent}

We wish to show certain chain conditions give rise to sets of marked groups which are $\Pi^1_1$ and not Borel in $\Gw$.  We begin by looking at the following chain condition.

\begin{defn}
A subgroup $H$ of $G$ is a \textbf{centralizer} in $G$ if $H=C_G(A)$ for some $A\subseteq G$.
\end{defn}

\begin{defn}
A group $G$ satisfies the \textbf{minimal condition on centralizers} if there is no strictly decreasing infinite chain $C_0 > C_1 > \ldots$ of centralizers in $G$.  We denote the class of countable groups satisfying the minimal condition on centralizers by $\mc M_C$.
\end{defn}

The class $\mc M_C$ is large, containing abelian groups, linear groups, and finitely generated abelian-by-nilpotent groups; see \cite{Br79} for further discussion. It is not hard to check that a group $G$ satisfies the minimal condition on centralizers if and only if it satisfies the maximal condition on centralizers, but our analysis is easier if we think about the minimal version of the chain condition.

Given a group $G\in\Gw$, we construct a tree $T_G\subseteq\wbaire$ and associated groups $G_s\in \Gw$ for each $s\in T_G$.  Each $G_s$ will be a centralizer in $G$.

\begin{enumerate}[$\bullet$]
\item Put $\emptyset\in T_G$ and let $G_\emptyset:=G=C_G(\emptyset)$.
\item Suppose that $s\in T_G$ and $G_s=C_G(\{g_s\})$ has already been defined.  If $C_G(\{g_s\}\cup\{g_i\})\neq C_G(\{g_s\})$, then put $s\conc i\in T_G$ and $G_{s\conc i}:=C_G(\{g_s\}\cup\{g_i\})$.
\end{enumerate}

\begin{lem}\label{lem:CentMapBorel}
The map $\Phi_C\from\Gw\to Tr$ given by $G\mapsto T_G$ is Borel.
\end{lem}
\noindent Intuitively, Lemma~\rm\ref{lem:CentMapBorel} holds since our construction is explicit; we delay a rigorous proof until Section \ref{sec:Borel}.

\begin{lem}\label{lem:CentMapReduction}
$T_G$ is well-founded if and only if $G\in\mc M_C$.
\end{lem}
\begin{proof}
If $G\in\mc M_C$, then $T_G$ contains no infinite branches by definition. If $G\notin\mc M_C$, then there is some infinite $A\subseteq G$ such that for all finite $B\subseteq A$, $C_G(A)\neq C_G(B)$.  Let $a_0<a_1<a_2<\ldots$ be such that $A=\{g_{a_0},g_{a_1},\ldots\}$.  By moving to a subsequence if necessary, we may assume that $C_G(\{g_{a_0},\ldots,g_{a_n}\}) \gneq C_G(\{g_{a_0},\ldots, g_{a_{n+1}}\})$ for all $n\in\Nb$.  Then $(a_0,\ldots,a_n)\in T_G$ for all $n\in\Nb$, so $T_G$ has an infinite branch.
\end{proof}

\begin{lem}\label{lem:CentSubRank}
Let $H,G\in\Gw$.  If $H\hookrightarrow G$, then $\rho(T_H)\leq\rho(T_G)$.
\end{lem}
\begin{proof}
Let $\alpha\from H\hookrightarrow G$ and let $\psi\from\Nb\to\Nb$ be such that $\alpha(h_k)=g_{\psi(k)}$.  We now define a map $\phi:T_H\rightarrow \wbaire$:  Let $\phi(\emptyset)=\emptyset$.  If $s\in T_H$ and $s=(s_0,\ldots,s_n)$, let $\phi(s)=(\psi(s_0),\ldots,\psi(s_n))$.  Clearly $\phi$ is monotone.  Further, if $s\in T_H$, then $H_s = C_H(\{h_s\}) \hookrightarrow C_G(\{g_{\phi(s)}\})$.  Since $C_G(\{g_{\phi(s)}\})\cap \alpha(H) \cong C_H(\{h_s\})$, we have that $C_G(\{g_{\phi(s\restriction k)}\}) \neq C_G(\{g_{\phi(s\restriction (k+1))}\})$ for all $k<|s|$.  Thus $\phi(s)\in T_G$.  It follows $\phi(T_H)\subseteq T_G$, and by Lemma \ref{lem:TrRkMonotone}, $\rho(T_H)\leq\rho(T_G)$.
\end{proof}

\begin{cor}\label{cor:CentIsoInv}
If $G,G'\in\Gw$ and $G\cong G'$, then $\rho(T_G)=\rho(T_{G'})$.
\end{cor}

We thus see that $\rho(T_G)$ is an isomorphism invariant, so it makes sense to talk about the rank of a group $G$ with the minimal condition on centralizers, even when not considering a specific marking.
\begin{defn} 
If $G$ has the minimal condition on centralizers, then $\rho(T_G)$ for some (any) marking of $G$ is called the \textbf{centralizer rank} of $G$.
\end{defn}

We also mention that the above results, except for Lemma \ref{lem:CentMapBorel}, work with arbitrary enumerations of the group $G$, not just those that can arise from viewing $G$ as a marked group.  Certain enumerations may be easier to use to calculate $\rho(T_G)$, and Corollary \ref{cor:CentIsoInv} assures us that using these enumerations will not affect the answer.  The same will be true of our later constructions.  Of course, in this paper Lemma \ref{lem:CentMapBorel} and analogous results are of central importance, so we will continue to work with groups as elements of $\Gw$.

We now argue the centralizer rank is unbounded below $\omega_1$.

\begin{lem}\label{lem:CentRankSucc}
For $A,B\in \mc{M}_C$ with $A$ nonabelian, $A\times B\in \mc{M}_C$ and $\rho(T_B)<\rho(T_{A\times B})$.
\end{lem}
\begin{proof}
It is easy to see $A\times B\in \mc{M}_C$. Let $a\in A$ be noncentral.  Then 
\[
C_{A\times B}(\{(a,e)\})=(A\times B)_i
\]
for some $i\in T_{A\times B}$ since the centralizer is not all of $A\times B$.  Further,
$$B\cong\{e\}\times B\leq C_{A\times B}(\{(a,e)\}),$$
so by Lemma \ref{lem:CentSubRank}, $\rho((T_{A\times B})_i)=\rho(T_{(A\times B)_i})\geq\rho(T_B)$.  The result now follows.
\end{proof}

\begin{lem}\label{lem:CentRankLim}
Let $\{A_i\}_{i\in\Nb}$ be countable groups.  If $A_i\in\mc M_C$ for all $i\in\Nb$, then there is a group $A\in \mc M_C$ such that $\rho(T_A)\geq\rho(T_{A_i})$ for all $i\in\Nb$.
\end{lem}
\begin{proof}
Let $A=\ast_{i\in\Nb} A_i$.  By \cite[Corollary 4.1.6]{MKS66}, which says that centralizers in free products are cyclic or centralizers of a conjugate of a free factor, we infer that $A\in\mc M_C$. Lemma \ref{lem:CentSubRank} now implies that $\rho(T_A)\geq\rho(T_{A_i})$ for all $i\in\Nb$, as desired.
\end{proof}

\begin{lem}\label{lem:CentRankUnbdd}
For all $\alpha<\omega_1$, there is $G\in\mc M_C$ such that $\rho(T_G)\geq\alpha$.
\end{lem}
\begin{proof}
We prove this inductively.  Clearly the lemma holds for $\alpha=0$.  Suppose $\alpha=\beta+1$ and the lemma holds for $\beta$.  Let $G\in\mc M_C$ be such that $\rho(T_G)\geq\beta$ and $A\in\mc M_C$ be nonabelian.  Applying Lemma~\rm\ref{lem:CentRankSucc}, we see $\rho(T_{A\times G})\geq\beta+1$.

Suppose $\alpha$ is a limit ordinal.  Since $\alpha$ is countable, there is a countable increasing sequence of $\alpha_i<\alpha$ such that $\sup_{i\in\Nb} \alpha_i =\alpha$.  Let $G_i\in\mc M_C$ be such that $\rho(T_{G_i})>\alpha_i$.  Applying Lemma~\rm\ref{lem:CentRankLim}, there is some $G\in\mc M_C$ such that $\rho(T_G)>\alpha_i$ for all $i\in\Nb$.  It now follows that $\rho(T_G)\geq\alpha$.
\end{proof}

\begin{lem}\label{lem:CentFG}
For all $\alpha<\omega_1$, there is a finitely generated $G\in\mc M_C$ such that $\rho(T_G)\geq\alpha$.
\end{lem}
\begin{proof}
Let $H\in\mc M_C$ be a group such that $\rho(T_H)\geq\alpha$.  Then \cite[Corollary on pg. 949]{KS71} implies that $H$ embeds into a 3-generated group $G\in\mc M_C$.  By Lemma \ref{lem:CentSubRank}, $\rho(T_G)\geq\rho(T_H)\geq\alpha$ verifying the lemma.
\end{proof}

We remark that the proof of the result cited in the previous uses nothing more complicated than free products with amalgamation and is similar to the classical Higman-Neumann-Neumann embedding result \cite{HNN49}.

\begin{thm}\label{thm:MCNotBorel}
$\mc M_C$ is $\Pi^1_1$ and not Borel in $\Gw$, and $\mc M_C\cap\Gw_{fg}$ is $\Pi^1_1$ and not Borel in $\Gw_{fg}$.
\end{thm}
\begin{proof}
Let $\Phi_C$ be the Borel map from Lemma~\rm\ref{lem:CentMapBorel}.  By Lemma \ref{lem:CentMapReduction}, $\Phi_C^{-1}(WF)=\mc M_C$, and since $\Phi_C$ is Borel, $\mc M_C$ is $\Pi^1_1$. Lemma \ref{lem:CentRankUnbdd} implies the ranks of the trees in $\Phi_C(\mc M_C)$ are unbounded below $\omega_1$, so the $\Pi_1^1$-rank on $\mc M_C$ given by $G\mapsto \rho(\Phi_C(G))$ is unbounded below $\omega_1$.  By Theorem \ref{thm:BddnessThm}, we conclude that $\mc M_C$ is not Borel. Lemma \ref{lem:CentFG} implies the ranks of the trees in $\Phi_C(\mc M_C\cap\Gw_{fg})$ are also unbounded below $\omega_1$, and by Theorem \ref{thm:BddnessThm}, we conclude that $\mc M_C\cap\Gw_{fg}$ is also not Borel.
\end{proof}

\section{The maximal condition on subgroups}\label{sec:Max}

We next consider a more basic chain condition.  Proving the analogue of Lemma \ref{lem:CentRankLim} in this context is more complicated, which is why we present it after the previous section.

\begin{defn}
A group $G$ satisfies the \textbf{maximal condition on subgroups}, abbreviated by saying a group satisfies max, if there is no strictly increasing chain $H_0 < H_1 < H_2 < \ldots$ of subgroups of $G$.  Equivalently, a group $G$ satisfies max if all of its subgroups are finitely generated.  We denote the class of groups satisfying max as $\mc M_{\max{}}$.
\end{defn}

Given a group $G\in \Gw$, we construct a tree $T_G\subseteq\wbaire$ and associated groups $G_s\in \Gw$ for each $s\in T_G$.

\begin{enumerate}[$\bullet$]
\item Put $\emptyset\in T_G$ and let $G_\emptyset:=\{e\}$.
\item Suppose that $s\in T_G$ and $G_s=\<\{g_s\}\>$ has already been defined.  If $\<\{g_s\}\cup\{g_i\}\>\neq\<\{g_s\}\>$, then put $s\conc i\in T_G$ and $G_{s\conc i}:=\<\{g_s\}\cup\{g_i\}\>$.
\end{enumerate}

\begin{lem}\label{lem:MaxMapBorel}
The map $\Phi_M\from\Gw\to Tr$ given by $G\mapsto T_G$ is Borel.
\end{lem}
We will prove Lemma~\ref{lem:MaxMapBorel} in Section \ref{sec:Borel}.

\begin{lem}\label{lem:MaxMapReduction}
$T_G$ is well-founded if and only if $G\in \mc M_{\max{}}$.
\end{lem}
\begin{proof}
If $G\in \mc M_{\max{}}$, then $T_G$ contains no infinite branches by definition.  If $G\notin \mc M_{\max{}}$, then there is some infinitely generated subgroup $H\leq G$.  There is some increasing sequence $a_0<a_1<\ldots$ of natural numbers such that $H=\<g_{a_0},g_{a_1},\ldots\>$.  We may assume that $\<g_{a_0},\ldots,g_{a_n}\> \lneq \<g_{a_0},\ldots,g_{a_{n+1}}\>$ for all $n\in\Nb$.  Then $(a_0,\ldots,a_n)\in T_G$ for all $n\in\Nb$, so $T_G$ has an infinite branch.
\end{proof}

\begin{lem}\label{lem:MaxSubRank}
Let $H,G\in\Gw$.  If $H\hookrightarrow G$, then $\rho(T_H)\leq\rho(T_G)$.
\end{lem}
\begin{proof}
Let $\psi\from\Nb\to\Nb$ be such that $h_k=g_{\psi(k)}$.  We now define a map $\phi:T_H\rightarrow \wbaire$:  Let $\phi(\emptyset)=\emptyset$.  If $s\in\wbaire$ and $s=(s_0,\ldots,s_n)$, let $\phi(s)=(\psi(s_0),\ldots,\psi(s_n))$.  Clearly $\phi$ is monotone.  Furthermore, if $s\in T_H$, then $H_s \cong G_{\phi(s)}$, hence $\phi(T_H)\subseteq T_G$. Lemma \ref{lem:TrRkMonotone} now implies $\rho(T_H)\leq\rho(T_G)$.
\end{proof}

The previous lemma implies $\rho(T_G)$ is a group invariant.
\begin{cor}
If $G,G'\in\Gw$ and $G\cong G'$, then $\rho(T_G)=\rho(T_{G'})$.
\end{cor}

\begin{defn} 
If $G$ has the maximal condition on subgroups, then $\rho(T_G)$ for some (any) marking of $G$ is called the \textbf{subgroup rank} of $G$.
\end{defn}

\begin{lem}\label{lem:MaxRankSucc}
For all groups $G\in \mc M_{\max{}}$, $G\times\Zb\in \mc M_{\max{}}$ and $\rho(T_G)<\rho(T_{G\times\Zb})$.
\end{lem}
\begin{proof}
It is easy to see $G\times \Zb$ satisfies max. For the latter condition, let $G=\{g_0,g_1,\ldots\}$ and $G\times\Zb=\{a_0,a_1,\ldots\}$.  There is some $k\in\Nb$ such that $a_k=(e_G,z)$ where $\Zb=\<z\>$.  Let $\psi\from\Nb\to\Nb$ be defined such that $a_{\psi(m)}=(g_m,e_\Zb)$. The map $\phi:T_G\rightarrow \wbaire$ given by $(s_0,\ldots,s_n)\mapsto(\psi(s_0),\ldots,\psi(s_n))$ is clearly monotone, and further, $\phi(T_G)\subseteq (T_{G\times\Zb})_k$.  By Lemma \ref{lem:TrRkMonotone}, $\rho(T_G) \leq \rho((T_{G\times\Zb})_k)<\rho(T_{G\times\Zb})$.
\end{proof}

\begin{lem}\label{lem:MaxRankLim}
Let $\{A_i\}_{i\in\Nb}$ be countable groups.  If $A_i\in \mc M_{\max{}}$ for each $i\in\Nb$, then there is a group $A\in \mc M_{\max{}}$ such that $\rho(T_A)\geq\rho(T_{A_i})$ for all $i\in\Nb$.
\end{lem}
\begin{proof}
This is a consequence of \cite[Theorem 2]{Ol89} due to A. Y. Olshanskii. This result gives a 2-generated group $A$ containing each of the $A_i$ such that every proper subgroup of $A$ is either contained in a conjugate of some $A_i$, is infinite cyclic, or is infinite dihedral.  Thus if every subgroup of each $A_i$ is finitely generated, then every subgroup of $A$ is finitely generated, and so $A\in\mc M_{\max{}}$.  Since each $A_i$ is a subgroup of $A$, Lemma \ref{lem:MaxSubRank} implies that $\rho(T_A)\geq\rho(T_{A_i})$ for all $i\in\Nb$ as desired.
\end{proof}

\begin{lem}\label{lem:MaxRankUnbdd}
For all $\alpha<\omega_1$, there is $G\in \mc M_{\max{}}$ such that $\rho(T_G)\geq\alpha$.
\end{lem}
\begin{proof}
The proof is the same as that of Lemma \ref{lem:CentRankUnbdd}, with Lemmas \ref{lem:MaxRankSucc} and \ref{lem:MaxRankLim} referenced at the appropriate places.
\end{proof}

\begin{thm}
$\mc M_{\max{}}$ is $\Pi^1_1$ and not Borel in $\Gw$ and $\Gw_{fg}$.
\end{thm}
\begin{proof}
The proof is the same as that of Theorem \ref{thm:MCNotBorel} using Lemmas \ref{lem:MaxMapBorel} and \ref{lem:MaxRankUnbdd} where appropriate.  The statement is true for $\mc G_{fg}$ simply because $\mc M_{\max{}}\subseteq\mc G_{fg}$.
\end{proof}

\section{The maximal condition on normal subgroups}\label{sec:MaxN}

Given a group $G$ and a set $S\subseteq G$, we write $\<\<S\>\>_G$ to denote the normal closure of $S$ in $G$.  We suppress the subscript $G$ when the group is clear from context.

\begin{defn}
A group $G$ satisfies the \textbf{maximal condition on normal subgroups}, abbreviated by saying a group satisfies max-n, if there is no infinite strictly increasing chain of normal subgroups of $G$.  Equivalently, a group $G$ satisfies max-n if each of its normal subgroups is the normal closure of finitely many elements of $G$.  We denote the class of groups satisfying max-n as $\mc M_n$.
\end{defn}

Given $G\in \Gw$, we construct a tree $T_G\subseteq\wbaire$ and associated groups $G_s\in \Gw$ for each $s\in T_G$.

\begin{enumerate}[$\bullet$]
\item Put $\emptyset\in T_G$ and let $G_\emptyset:=G$.
\item Suppose that $s\in T_G$ and $G_s=G/\ngrp{\{g_s\}}$ has already been defined.  If $\ngrp{\{g_s\}\cup\{g_i\}}\neq \ngrp{\{g_s\}}$, then put $s\conc i\in T_G$ and $G_{s\conc i}:=G/\ngrp{\{g_s\}\cup\{g_i\}}$.
\end{enumerate}

\begin{lem}\label{lem:MaxNMapBorel}
The map $\Phi_{M_n}\from\Gw\to Tr$ given by $G\mapsto T_G$ is Borel.
\end{lem}
We prove Lemma~\ref{lem:MaxNMapBorel} in Section \ref{sec:Borel}.

\begin{lem}\label{lem:MaxNMapReduction}
$T_G$ is well-founded if and only if $G\in\mc M_n$.
\end{lem}
\begin{proof}
If $G\in\mc M_n$, then $T_G$ contains no infinite branches by definition.  If $G\notin\mc M_n$, then there is a normal subgroup $N\normal G$ such that $N=\ngrp{g_{a_0},g_{a_1},\ldots}$ and
$$\ngrp{g_{a_0},\ldots,g_{a_n}} \lneq \ngrp{g_{a_0},\ldots,g_{a_{n+1}}}$$
for all $n\in\Nb$.  Thus the sequence $(a_0,\ldots, a_n)$ is an element of $T_G$ for any $n\in\Nb$, so $T_G$ has an infinite branch.
\end{proof}

\begin{lem}\label{lem:MaxNImRank}
If $G\in\mc M_n$ and $f:G\surject G'$, then  $\rho(T_{G})\geq \rho(T_{G'})$ with equality if and only if $f$ is injective.
\end{lem}
\begin{proof}
Since $G$ is max-n, $\ker(f)=\ngrp{S}$ for some finite $S=\{g_{s_0},\ldots,g_{s_n}\}$. We may assume that $n$ is minimal, so no element of $S$ is in the normal closure of the others. Setting $(s_0,\dots,s_n)=:s\in\wbaire$, the minimality of $S$ implies $s\in T_G$; in the case $\ker(f)=\{1\}$, we take $s=\emptyset$. Let $\psi\from\Nb\to\Nb$ be a map such that $g'_k=f(g_{\psi(k)})$.  Then for all $i_0,\ldots,i_k\in\Nb$, 
\[
G'/\<\<g'_{i_0},\ldots,g'_{i_k}\>\>_{G'} \cong G/\<\<g_{\psi(i_0)},\ldots,g_{\psi(i_k)},S\>\>_G,
\]
and the monotone map $\phi:T_{G'}\rightarrow \wbaire$ given by $(r_0,\ldots,r_n)\mapsto(\psi(r_0),\ldots,\psi(r_n))$ sends $T_{G'}$ into $(T_G)_s$.  By Lemma \ref{lem:TrRkMonotone}, $\rho(T_{G'}) \leq \rho((T_G)_s)\leq\rho(T_G)$, and the rightmost inequality is strict if and only if $s\neq \emptyset$.
\end{proof}

We conclude this rank is also isomorphism invariant.

\begin{cor}\label{cor:MaxNRankIsoInv}
If $G,G'\in\Gw$ and $G\cong G'$, then $\rho(T_G)=\rho(T_{G'})$.
\end{cor}

Recall that a group is \textbf{hopfian} if it is not isomorphic to any of its proper quotients.  The following corollary is easy enough to prove directly, but it follows immediately from Lemma \ref{lem:MaxNImRank} and Corollary~\ref{cor:MaxNRankIsoInv}.

\begin{cor}
If $G\in\mc M_n$, then $G$ is hopfian.
\end{cor}

Unlike the previous invariants, this rank has appeared before in the literature; cf.\ \cite{C11}. 

\begin{defn} 
If $G$ has the maximal condition on normal subgroups, then $\rho(T_G)$ for some (any) marking of $G$ is called the \textbf{length} of $G$.
\end{defn}

If we were to follow our template from previous sections, we would move on to analogues of Lemmas \ref{lem:CentRankSucc} and \ref{lem:CentRankLim}.  However, we were unable to prove an analogue of Lemma \ref{lem:CentRankLim} which would take advantage of Lemma \ref{lem:MaxNImRank}.  Such a result would be a sort of dual version of the result of Olshanskii cited in the proof of Lemma \ref{lem:MaxRankLim}.  Specifically, the following question is open to the best of the authors' knowledge:

\begin{quest}
Suppose $\{A_i\}_{i\in\Nb}$ is a set of normally $k$-generated max-n groups. Is there a max-n group $A$ such that $A\twoheadrightarrow A_i$ for all $i\in\Nb$?
\end{quest}

A positive answer to this question would give us exactly the right analogue of Lemma \ref{lem:CentRankLim}.  Lacking this, we will use a construction involving (restricted) wreath products. Recall the wreath product of $H$ and $G$ is $H\wr G:=H^{<G}\rtimes G$ where $H^{<G}$ denotes the direct sum and $G\acts H^{<G}$ by shift; in the case $G\acts X$ for some set $X$, we write $H\wr_X G:=H^{<X}\rtimes G$. We will see that we can relate $\rho(T_{H\wr G})$ to both $\rho(T_H)$ and $\rho(T_G)$, while Lemma \ref{lem:MaxNImRank} alone only gives us information about how $\rho(T_{H\wr G})$ and $\rho(T_G)$ relate.

We will focus on perfect max-n groups with no central factors; let us call the set of such groups $\mc M'_n$. A group $G$ is said to have a \textbf{central factor} if there are normal subgroups $L\trianglelefteq M $ in $G$ such that $M/L$ is nontrivial and central in $G/L$. Since $\mc M'_n\subseteq \mc M_n$, it is enough for our purposes to show that $\rho$ is unbounded below $\omega_1$ on $\mc M'_n$ and $\mc M'_n\cap \Gw_{fg}$.

\begin{lem}\label{lem:MaxNRankSucc}
Let $S$ be an infinite simple group.  For all groups $G\in\mc M_n$, $G\times S\in\mc M_n$ and $\rho(T_G)<\rho(T_{G\times S})$.  If $G\in\mc M'_n$, then so is $G\times S$.
\end{lem}
\begin{proof}
It is easy to see that $G\times S\in\mc M_n$, and since $G$ is a quotient of $G\times S$, Lemma \ref{lem:MaxNImRank} implies $\rho(T_G)<\rho(T_{G\times S})$.  If $G$ is perfect, then $G\times S$ is perfect, so for the last statement we need only to check that if $G$ has no central factors, then $G\times S$ has no central factors.  Suppose that $L, M\normal G\times S$ give a central factor. Let $\pi\from G\times S \to G$ be the usual projection.  Since $G$ has no central factors, $\pi(M)=\pi(L)$.  Thus $MS=LS$, so $M=L(S\cap M)$.  Since $S$ has no central factors, $S\cap M=S\cap L$.  We conclude that $M=L$, whereby $G\times S$ has no central factors, a contradiction.
\end{proof}

Lemma~\rm\ref{lem:MaxNRankSucc} allows us to find a group in $\mc M'_n$ with rank greater than a given group in $\mc M'_n$.  However, we also need to be able to find a group in $\mc M'_n$ with rank greater than a countable family of groups from $\mc M'_n$.  We begin by looking at properties of the ranks of wreath products.

\begin{lem}\label{lem:max-n}
Suppose $H$ and $G$ are groups satisfying max-n. Then $\rho(T_{H\wr G})\geq \rho(T_G)+\rho(T_H)$. 
\end{lem}
\begin{proof} 
For each $h\in H$ define $f_h\in H^{<G}$ by
\[
f_h(g)=
\begin{cases}
h, & \text{ if } g=e\\
e, & \text { else.}
\end{cases}
\]
Let $H=\{h_0,h_1,\ldots\}$ and let $\psi\from\Nb\to\Nb$ be a map such that $f_{h_i}=g_{\psi(i)}$.  We now define a monotone $\phi:T_H\rightarrow \wbaire$:  Put $\phi(\emptyset)=\emptyset$.  For non-empty $s\in T_H$, define $\phi$ by 
$$(s_0,\dots,s_k)\mapsto \left(\psi(s_0),\dots, \psi(s_k)\right).$$

\indent We argue $\phi$ maps $T_H$ into $T_{H\wr G}$ by induction on the length of $s\in T_H$. As the base case is immediate, say $s\in T_H$ and $s\conc k\in T_H$. By construction, it is the case that $\ngrp{\{h_s\}\cup \{h_k\}}_H\neq\ngrp{\{h_s\}}_H$.  For all $t\in T_H$, $\ngrp{\{g_{\phi(t)}\}}_{H\wr G}=\ngrp{\{h_t\}}_H^{<G}$, hence 
\[
\ngrp{\{g_{\phi(s)}\}\cup\{g_{\psi(k)}\}}_{H\wr G}\neq \ngrp{\{g_{\phi(s)}\}}_{H\wr G}.
\]
We conclude that $\phi(s\conc k)\in T_{H\wr G}$, so $\phi$ maps $T_H$ into $T_{H\wr G}$.\par

Now if $s=(s_0,\ldots,s_n)\in T_H$ is a terminal node, then $\ngrp{\{h_s\}}_H = H$.  In this case $(H\wr G)/\ngrp{\{g_{\phi(s)}\}}_{H\wr G} \cong G$, so $\rho(T_G)=\rho((T_{H\wr G})_{\phi(s)})$ by Corollary \ref{cor:MaxNRankIsoInv}.  The desired result now follows.
\end{proof}

In general, $H\wr G$ need not be max-n. A theorem of P. Hall provides a sufficient condition for this.

\begin{thm}[Hall, {\cite[Theorem 4]{H54}}]\label{thm:max-n_wreath}
Let $H$ and $G$ be groups satisfying max-n. If $H$ has no central factors, then $H\wr G$ satisfies max-n.
\end{thm}

Our next lemma allows us to iterate wreath products and remain in $\mc M'_n$.

\begin{lem}\label{lem:central_factors}
If $G$ and $H$ have no central factors, then $H\wr G$ has no central factors.
\end{lem}
\begin{proof}
Suppose $L\trianglelefteq M$ gives a central factor of $H\wr G$. Let $\pi:H\wr G\rightarrow G$ be the usual projection. Since $G$ has no central factors, it must be the case that $\pi(L)=\pi(M)$, so $LH^{<G}=MH^{<G}$. Thus, $M=L(H^{<G}\cap M)$, and it suffices to show $H^{<G}\cap M\leq L$. Since $H$ has no central factors, it follows similarly to the proof of Lemma~\rm\ref{lem:MaxNRankSucc} that $H^{F}\cap M=H^{F}\cap L$ for all finite $F\subseteq G$. We conclude that $H^{<G}\cap M\leq L$ verifying the lemma.
\end{proof}

It is easy to see the wreath product of two perfect groups is perfect, so using Theorem \ref{thm:max-n_wreath} and Lemma \ref{lem:central_factors}, the class $\mc M'_n$ is closed under wreath products. With the following fact from the literature, we are equipped to prove the desired lemma.
\begin{lem}[{\cite[Lemma 3.6]{GKO14}}]\label{lem:wreath}
Suppose $A,B$ are countable groups and form $G=A\wr B$.  If $N\trianglelefteq G$ meets $B$ non-trivially, then $[A,A]^{<B}\leq N$.
\end{lem}

\begin{lem}\label{lem:MaxNRankLim}
Let $\{A_i\}_{i\in\Nb}$ be countable groups.  If $A_i\in\mc M'_n$ for all $i\in\Nb$, then there is a group $A\in\mc M'_n$ such that $\rho(T_{A_i})\leq\rho(T_A)$ for all $i\in\Nb$.
\end{lem}
\begin{proof}
For each $n$, put $G_n:=A_n\wr\left(\dots \wr A_0\right)$.  By making the natural identification, we may assume $G_{n}\leq G_{n+1}$ for all $n$ and form $A:=\bigcup_{n\in \Nb} G_n$.  (Alternatively, one may take the direct limit.) \par

\indent Consider $N\normal A$.  Certainly, $N\cap G_n$ is non-trivial for some $n$.  Fix such an $n$ and take $k>n$. We now see $N\cap G_{k}\trianglelefteq G_{k}=A_{k}\wr G_{k-1}$ is a normal subgroup that meets $G_{k-1}$ non-trivially.  Applying Lemma~\rm\ref{lem:wreath}, $[A_{k},A_{k}]^{<G_{k-1}}\leq N\cap G_{k}$. Since $A_{k}$ is perfect, we have that $A_{k}^{<G_{k-1}}\leq N$. It now follows that $A/N$ is isomorphic to a quotient of $G_n$.

\indent Suppose $(N_i)_{i\in \Nb}$ is an increasing sequence of normal subgroups of $A$. By the previous paragraph, $A/N_0$ is a quotient of $G_n$ for some $n$, and Theorem~\rm\ref{thm:max-n_wreath} implies that each $G_n$ is a max-n group. Letting $\pi:A\rightarrow A/N_0$ be the usual projection, it is thus the case that $\pi(N_i)=\pi(N_j)$ for all sufficiently large $i$ and $j$. Therefore, $N_i=N_j$ for all sufficiently large $i$ and $j$, and $A$ satisfies max-$n$. \par

\indent For each $n$ and $k>n$, define 
\[
L_n^k:=A_k\wr_{G_{k-1}}\left( \dots\wr_{G_{n+2}}\left(A_{n+2}\wr_{G_{n+1}}A_{n+1}^{<G_n}\right)\right)
\]
and put $L_n:=\bigcup_{k>n}L_n^k$. We see $L_n\trianglelefteq A$ and $A/L_n\cong G_n$.  By Lemmas \ref{lem:max-n} and \ref{lem:MaxNImRank}, $\rho(T_{A_n})\leq\rho(T_{G_n})\leq\rho(T_A)$ for all $n$.

\indent We finally verify $A$ is perfect and has no central factors. That $A$ is perfect is immediate. It follows from Lemma~\rm\ref{lem:central_factors} and induction that each $G_n$ has no central factors. Since any factor of $A$ is a factor of $G_n$ for some $n$, $A$ has no central factors.
\end{proof}

\begin{lem}\label{lem:MaxNRankUnbdd}
For all $\alpha<\omega_1$, there is $G\in\mc M'_n$ such that $\rho(T_G)\geq\alpha$.
\end{lem}
\begin{proof}
The proof is the same as that of Lemma \ref{lem:CentRankUnbdd}, with Lemmas \ref{lem:MaxNRankSucc} and \ref{lem:MaxNRankLim} referenced at the appropriate places.
\end{proof}

The groups given by Lemma~\rm\ref{lem:MaxNRankUnbdd} are not, in general, finitely generated. For finding finitely generated examples another result of Hall is needed.
\begin{lem}[Hall, {\cite[cf. Theorem 4]{H61}}]\label{lem:Hall} 
Let $H$ be a countable group. Then there exists a short exact sequence 
\[
\{e\}\rightarrow M\rightarrow G\rightarrow \Zb\rightarrow \{e\}
\]
where $G$ is 2-generated, $[M,M]=[H,H]^{<\Zb}$, and there is $t\in G$ so that the conjugation action of $t$ on $[M,M]$ is by unit shift. 
\end{lem}

It is useful to sketch Hall's construction of $G$. Let $\{h_i\}_{i\in \Nb}$ list $H$ and form the unrestricted wreath product $H^{\Zb}\rtimes \Zb$. Define $\sigma\in H^{\Zb}$ by
\[
\sigma(i):=
\begin{cases}
h_n, & \text{ if }i=2^n\\
e, & \text{ else}
\end{cases}
\]
and let $t$ be a generator for $\Zb$ in $H^{\Zb}\rtimes \Zb$. The desired group is then $G:=\grp{t,\sigma}$. The subgroup $M$ equals $\grp{g\sigma g^{-1}\mid g\in G}$.\par

\indent We point out a consequence of the construction for later use: Suppose $H$ is perfect and $h\in H$. Taking $f_h\in [H,H]^{<\Zb}=H^{<\Zb}$ as defined in Lemma \ref{lem:max-n}, the construction of $G$ implies $\ngrp{f_h}_G=\ngrp{h}_H^{<\Zb}$. 

\begin{cor}\label{cor:MaxNFG}
For each $\alpha<\omega_1$, there is a finitely generated group $G\in\mc M_n$ with \mbox{$\rho(T_G)\geq \alpha$}.
\end{cor}

\begin{proof}
Fix $\alpha<\omega_1$ and apply Lemma~\rm\ref{lem:MaxNRankUnbdd} to find a group $H\in\mc M'_n$ with $\rho(T_H)\geq \alpha$. We now apply Lemma~\rm\ref{lem:Hall} to find a 2-generated group $G$ with a short exact sequence 
\[
\{e\}\rightarrow M\rightarrow G\rightarrow \Zb\rightarrow \{e\}
\]
where $[M,M]=[H,H]^{<\Zb}=H^{<\Zb}$. \par

\indent The group $G/[M,M]$ is a finitely generated metabelian group, hence it satisfies max-n by \cite[Theorem 3]{H54}. On the other hand, any normal subgroup of $G$ that lies in $[M,M]=H^{<\Zb}$ is shift-invariant because $G$ contains an element that acts by shift on $[M,M]$. Since $H\wr\Zb$ is max-n, it follows that $H^{<\Zb}$ is max-$G$; that is to say $H^{<\Zb}$ has the maximal condition on subgroups invariant under the conjugation action by $G$. We conclude the group $G$ is max-n.\par

\indent It remains to compute a lower bound for $\rho(T_G)$. Using again the notation from Lemma \ref{lem:max-n}, find $\psi:\Nb\rightarrow \Nb$ such that for each $k\in \Nb$ we have $f_{h_k}=g_{\psi(k)}$. Since $\ngrp{f_h}_G=\ngrp{h}_H^{<\Zb}$, we may argue as in Lemma~\ref{lem:max-n} to conclude that $\alpha\leq \rho(T_H)\leq \rho(T_G)$. That is to say, we can define a monotone $\phi\from T_H\to T_G$ and by Lemma~\rm\ref{lem:TrRkMonotone} conclude that $\alpha\leq \rho(T_H)\leq \rho(T_G)$.
\end{proof}

\begin{thm}
$\mc M_n$ is $\Pi^1_1$ and not Borel in $\Gw$, and $\mc M_n\cap G_{fg}$ is $\Pi^1_1$ and not Borel in $\mc G_{fg}$.
\end{thm}
\begin{proof}
This follows from Theorem \ref{thm:BddnessThm}, Lemma \ref{lem:MaxNMapBorel}, and Corollary \ref{cor:MaxNFG}.
\end{proof}

\section{Elementary amenable groups}\label{sec:EAGroups}

Perhaps surprisingly, the property of being elementary amenable may also be characterized by well-founded trees. This in turn gives a chain condition equivalent to elementary amenability.

\subsection{Preliminaries} We study the collection of elementary amenable groups. This class is typically defined as follows:

\begin{defn}\label{def:EA}
The collection of \textbf{elementary amenable groups}, denoted $\EA$, is the smallest collection of countable groups such that
\begin{enumerate}[(i)]
\item $\EA$ contains all finite groups and abelian groups.
\item $\EA$ is closed under group extensions.
\item $\EA$ is closed under countable increasing unions.
\item $\EA$ is closed under taking subgroups.
\item $\EA$ is closed under taking quotients.
\end{enumerate}
\end{defn}

Our results here require a fairly well-known embedding result, which is based on a generalization of Lemma~\ref{lem:Hall}.
\begin{prop}[Hall; Neumann, Neumann {\cite[Theorem 5.1]{NN59}}]\label{prop:EAembedding}
Suppose $K\in \EA$. Then there exists $H\in \EA$ and a short exact sequence 
\[
\{e\}\rightarrow M\rightarrow G\rightarrow \Zb\rightarrow \{e\}
\]
where $G$ is 2-generated, $G\in \EA$, $[M,M]=[H,H]^{<\Zb}$, and $K$ embeds into $[H,H]$. 
\end{prop}  

\subsection{Decomposition trees}
We now define a tree associated to a marked group $G$. Just as in the previous sections, we will see that this tree being well-founded or not gives group-theoretic information about $G$, in this case characterizing being an elementary amenable group.\par

Let $G\in\Gw$. For $n\geq 0$, put $R_n(G):=\<g_0,\ldots,g_n\>$ and for $k\geq 1$, define 
\[
S_k(G):=[G,G]\cap\bigcap\mc{N}_{k}(G)
\]
where $\mc{N}_k(G):=\{N\trianglelefteq G\;|\;|G:N|\leq k+1\}$. For each $l\geq 1$, we now define a tree $T^l(G)\subseteq\wbaire$ and associated groups $G_s\in\Gw$ as follows:
\begin{enumerate}[$\bullet$]
\item Put $\emptyset\in T^l(G)$ and let $G_{\emptyset}:=G$.
\item Suppose we have $s\in T^l(G)$ and $G_s$. If $G_s\neq\{e\}$, put $s\conc n\in T^l(G)$ and  $G_{s\conc n}:=S_{|s|+l}\left(R_n\left(G_s\right)\right)$. 
\end{enumerate}

We call $T^l(G)$ the \textbf{decomposition tree} of $G$ with offset $l$. This tree is always non-empty, and if $s\in T^l(G)$ is terminal, then $G_s=\{e\}$.  Since the composition of the functions $R_n$ and $S_k$ is associative, we obtain a useful observation:

\begin{obs}\label{prop:child_rank}
For $s\in T^l(G)$, $T^l(G)_s=T^{|s|+l}(G_s)$, and for each $r\in T^{|s|+l}(G_s)$, $(G_{s})_r=G_{s\conc r}$ as marked groups. This implies, in particular, that if $T^l(G)$ is well-founded, then so is $T^{|s|+l}(G_s)$.
\end{obs}

\begin{lem}\label{lem:Phi_borel}
For each $l\geq 1$, the map $\Phi^l\from\Gw\to Tr$ given by $G\mapsto T^l(G)$ is Borel.
\end{lem}

As usual, we postpone the proof of this lemma to Section \ref{sec:Borel}. 

\begin{lem}\label{lem:xi_indp}
Let $G,H\in\Gw$ and $H\hookrightarrow G$. Then for all $l\geq k \geq 1$, 
\[
\rho\left(T^l(H)\right)\leq\rho\left(T^k(G)\right).
\] 
In particular, for $G,G'\in\Gw$, if $G\cong G'$, then
\[
\rho\left(T^l(G)\right)=\rho\left(T^l(G')\right)
\]
for all $l\geq 1$.
\end{lem}
\begin{proof}
We induct on $\rho(T^k(G))$ simultaneously for all $k$. If $\rho(T^k(G))=1$, then $G=\{e\}$, so $H=\{e\}$.  Suppose the lemma holds for all $G$ and $k$ with $\rho(T^k(G))\leq\beta$.  Suppose that $f\from H\to G$ is an embedding and $\rho(T^k(G))=\beta+1$.  For all $n\geq 0$, there is some $k(n)$ so that $f(R_n(H))\leq R_{k(n)}(G)$.  It follows that $f(H_n)\leq G_{k(n)}$ for all $n\geq 0$ since $S_{l+1}(G_{k(n)})\leq S_{k+1}(G_{k(n)})$.  By the inductive hypothesis and Observation~\ref{prop:child_rank},
\[
\rho\left(T^l(H)\right)=\sup_{n\in\Nb}\left\{\rho\left(T^{l+1}(H_n)\right)\right\}+1 \leq \sup_{n\in\Nb}\left\{\rho\left(T^{k+1}(G_{k(n)})\right)\right\}+1 \leq \rho\left(T^k(G)\right)
\]
completing the induction.
\end{proof}

\begin{cor}\label{cor:some_all}
 For $G\in \Gw$, $T^l(G)$ is well-founded for some $l\geq 1$ if and only if $T^l(G)$ is well-founded for all $l\geq 1$.
\end{cor}
\begin{proof} Suppose $G\in \Gw$ is so that $T^l(G)$ is well-founded. In view of Lemma~\ref{lem:xi_indp}, $T^k(G)$ is well-founded for all $k\geq l$. For $n\leq l$, take $s\in T^n(G)$ with $|s|=l$; if no such $s$ exists then $T^n(G)$ is plainly well-founded. There is an injection $G_s\hookrightarrow G$, so applying Lemma~\ref{lem:xi_indp} once again, 
\[
\rho\left(T^{n+|s|}(G_s)\right)\leq \rho\left(T^{n+|s|}(G)\right).
\]
By choice of $s$, $n+|s|\geq l$, hence $\rho(T^{n+|s|}(G))<\omega_1$. Since $T^n(G)_s=T^{n+|s|}(G_s)$, we conclude $T^{n}(G)_s$ is well-founded for each $s$ of length $l$. The tree $T^n(G)$ is therefore well-founded, and the corollary follows. 
\end{proof}

Define $\mathrm{W}:=\cup_{l=1}^{\infty} (\Phi^l)^{-1}(WF)$; that is, $\mathrm{W}$ is the collection of marked groups so that some decomposition tree is well-founded. By Corollary~\ref{cor:some_all}, every decomposition tree of a group in $\mathrm{W}$ is well-founded; that is to say, $W=\cap_{l=1}^{\infty} (\Phi^l)^{-1}(WF)$.\par

\indent Lemma~\ref{lem:xi_indp} shows the rank of a decomposition tree is independent of the marking. We thus define

\begin{defn} The \textbf{decomposition rank} of $G\in\mathrm{W}$ is defined to be
\[
\xi(G):=\min_{k\in \omega}\rho\left(T^k(G)\right)
\]
for some (any) marking of $G$. The \textbf{decomposition degree} is defined to be 
\[
\deg(G):=\min\left\{k \mid \xi(G)=\rho\left(T^k(G)\right)\right\}
\]
for some (any) marking of $G$.
\end{defn}

\begin{cor}\label{lem:sgrp_xi} 
If $G,H\in\Gw$ and $H\hookrightarrow G$, then $\xi(H)\leq \xi(G)$.
\end{cor}

\begin{rmk}
The decomposition rank in a fairly straightforward manner tracks the number of extensions and unions applied to produce the group.  The decomposition degree, on the other hand, is currently mysterious. It somehow tracks the size of the finite groups ``appearing" in the construction of an elementary amenable group. We do not consider the decomposition degree further as it is tangential to our goal. We do study the decomposition rank in detail.
\end{rmk}

We now show that $\mathrm{W} \subseteq \EA$ and $\mathrm{W}$ enjoys the same closure properties as $\EA$, so that in fact $\EA=\mathrm{W}$. 
\begin{thm}
If $G\in\mathrm{W}$, then $G\in\EA$.
\end{thm}
\begin{proof}
We induct on $\xi(G)$. For the base case, if $\xi(G)=1$, then $G=\{e\}$ and $G\in\EA$. Suppose the theorem holds for all $\alpha<\beta$ and $\xi(G)=\rho(T^l(G))=\beta$. Consider $R_i(G)$. Since $R_i(G)$ is finitely generated, $\mc{N}_{l+1}(R_i(G))$ is finite, so 
\[
\left|[R_i(G),R_i(G)]:G_i\right| <\infty.
\]
We infer $R_i(G)/G_i$ is finite-by-abelian and, therefore, elementary amenable.\par

\indent On the other hand, Observation~\rm\ref{prop:child_rank} gives $\rho(T^{1+l}(G_i))=\rho(T^l(G)_i)$. Hence, $\rho(T^{1+l}(G_i))<\beta$, and we conclude that $G_i\in \EA$ from the inductive hypothesis. As $\EA$ is closed under group extensions and countable increasing unions, $R_i(G)\in\EA$ for all $i\in\omega$, whereby $G\in\EA$.
\end{proof}

The family $\mathrm{W}$ also has the same closure properties as $\EA$. Lemma~\ref{lem:xi_indp} already shows $\mathrm{W}$ is closed under taking subgroups. For the other closure properties, we require several lemmas. 

\begin{lem}\label{lem:finandab}
$\mathrm{W}$ contains all finite groups and all abelian groups.
\end{lem}
\begin{proof}
If $G$ is abelian, then $\rho(T^1(G))\leq 2$.  If $G$ is finite with size $m$, then $\rho(T^m(G))\leq 2$.
\end{proof}

We next consider increasing unions.
\begin{lem}\label{lem:union_xi}
If $G=\cup_{i\in\Nb} H_i$ and each $H_i\in\mathrm{W}$, then $G\in\mathrm{W}$.
\end{lem}
\begin{proof}
For each $i\in\Nb$, let $\alpha_i:=\rho(T^{1}(H_i))<\omega_1$.  Since each $R_n(G)$ is finitely generated, there is some $m_n\in\Nb$ such that $R_n(G)\leq H_{m_n}$. By Lemma \ref{lem:xi_indp}, $\rho(T^{1}(G_n))\leq\rho(T^{1}(H_{m_n}))=\alpha_{m_n}$.  We conclude $\rho(T^1(G))\leq \sup_{i\in\Nb} (\alpha_{m_i}) + 1 < \omega_1$, and thereby, $G\in\mathrm{W}$.
\end{proof}

In our construction, given $G$ and $k\geq 1$, we are particularly interested in the $G_i$ associated with $i\in T^k(G)$.  We will see their decomposition rank is related to that of $G$ in a simple way; this observation is necessary for showing $W$ is closed under taking extensions and quotients.

\begin{lem}\label{lem:rk_xi}
Suppose $G\in\mathrm{W}$ is non-trivial and $\deg(G)=k$. Then
\[
\sup_{i\in \omega}\xi(G_i)+1\leq \xi(G)
\]
where $G_i$ is the subgroup of $G$ associated to $i\in T^k(G)$.  In particular, $\xi(G_i)<\xi(G)$ for all $i\in\Nb$.
\end{lem}
\begin{proof}
By construction, for all $i\in\Nb$,
\begin{align*}
\rho\left(T^{k+1}(G_i)\right)+1 &=\rho\left(T^k(G)_i\right)+1 \\
 &\leq \rho\left(T^k(G)\right).
\end{align*}
Hence,
\begin{align*}
\sup_{i\in \Nb}\xi(G_i)+1 &= \sup_{i\in \Nb} \left\{ \min_{l\in\Nb} \rho\left(T^l(G_i)\right)\right\}+1 \\
 &\leq \sup_{i\in\Nb} \left\{ \rho\left(T^{k+1}(G_i)\right) \right\} +1\\
 &=  \rho\left(T^k(G)\right) \\
 &= \xi(G)
\end{align*}
as desired.
\end{proof}
\noindent The inequality in Lemma~\rm\ref{lem:rk_xi} may be strict; for example, consider $\operatorname{Sym}_{fin}(\Nb)$, the group of finitely supported permutations of $\Nb$. We also point out that Lemma~\rm\ref{lem:rk_xi} \emph{does not} hold for choices of $k$ such that $\rho(T^k(G))\neq\xi(G)$.\par
 
We next show that $\mathrm{W}$ is closed under extensions.  We will first prove a weaker statement.  This approach is inspired by \cite{Os02}.

\begin{lem}\label{lem:finorab_xi}
Suppose that $N\in\mathrm{W}$, $B$ is finite or abelian, and there is a short exact sequence
\[
1 \rightarrow N \rightarrow G \rightarrow B \rightarrow 1 .
\]
Then $G\in\mathrm{W}$.
\end{lem}
\begin{proof}
Suppose first that $B$ is abelian. Thus, $[G,G] \leq N$, so for any $l\geq 1$ and all $n\in T^l(G)$, $G_n\leq N$.  It follows from Lemma \ref{lem:xi_indp} that for all $n\in\Nb$, 
\[
\rho(T^{l+1}(G_n))\leq\rho(T^{l+1}(N))<\omega_1.
\]
Appealing to Observation \ref{prop:child_rank}, we infer $\rho(T^l(G))<\omega_1$, so $G\in \mathrm{W}$.

Suppose that $B$ is finite and $|G\colon N|=k$.  For all $n\in T^k(G)$, $G_n\leq N$, so as above, $T^k(G)$ is well-founded. Hence, $G\in W$.
\end{proof}

\begin{lem}\label{lem:ext_xi}
Suppose the group $G$ is the extension of a group $B\in\mathrm{W}$ by a group $N\in\mathrm{W}$.  Then $G\in\mathrm{W}$. The family $\mathrm{W}$ is thus closed under group extensions.
\end{lem}
\begin{proof}
We first establish the following claim.

\begin{claim*}
If $N\in\mathrm{W}$ and $B$ is finite-by-abelian, then the extension of $B$ by $N$ is in $\mathrm{W}$.
\end{claim*}
\begin{proof}[Proof of claim.]
Suppose that $B$ is the extension of an abelian group $A$ by a finite group $F$.  Let $F_0$ be the preimage of $F$ in $G$.  Then $G/F_0 \cong B/F \cong A$, so $G/F_0$ is abelian.  Since $F_0$ is the extension of the finite group $F$ by $N$, Lemma \ref{lem:finorab_xi} implies that $F_0\in\mathrm{W}$.  Applying Lemma \ref{lem:finorab_xi} a second time, $G\in\mathrm{W}$.
\end{proof}

We now prove the lemma by induction on $\beta=\xi(B)$.  If $\beta=1$, then $B=\{e\}$ and the induction claim holds trivially.  Suppose the result holds for all $\delta<\beta$.  First, assume that $B$ is finitely generated, let $\deg(B)=l$, and form the decomposition tree $T^l(B)$. By finite generation, there is some $m\in\Nb$ such that for all $k\geq m$, $R_k(B)=B$, so $B_k=[B,B]\cap\bigcap\mc N_l(B)$.  We now consider $K\normal G$ the preimage of $B_k$ under the projection map. The group $K$ is the extension of $B_k$ by $N$, and $\xi(B_k)<\xi(B)$ by Lemma~\ref{lem:rk_xi}. The inductive hypothesis therefore implies $K\in\mathrm{W}$. On the other hand, $G/K$ is finite-by-abelian, so $G\in\mathrm{W}$ by our claim.\par

\indent If $B$ is not finitely generated, then $B=\cup_{n\in\Nb} R_n(B)$, and $\xi(R_n(B))\leq\xi(B)$ for all $n\in\Nb$.  Letting $C_n$ be the preimage in $G$ of $R_n(B)$, the previous paragraph implies $C_n\in\mathrm{W}$. Since $G=\cup_{n\in\Nb} C_n$, Lemma \ref{lem:union_xi} ensures that $G\in\mathrm{W}$.
\end{proof}

Finally, we show that $W$ is closed under quotients.

\begin{lem}\label{lem:xi_quot} 
If $G\in\mathrm{W}$ and $L\normal G$, then $G/L\in \mathrm{W}$. 
\end{lem}

\begin{proof} We argue by induction on $\xi(G)$.  As the base case is immediate, suppose the lemma holds up to $\beta$ and let $G$ be such that $\xi(G)=\beta+1$. In view of Lemma~\ref{lem:union_xi}, we may assume $G$ is finitely generated, so $R_n(G)=G$ for all suitably large $n$. Say $k=\deg(G)$ and let $G_n$ be the subgroup corresponding to $n\in T^k(G)$.\par
 
\indent By the inductive hypothesis and Lemma~\ref{lem:rk_xi}, $G_nL/L\cong G_n/G_n\cap L\in \mathrm{W}$ for each $n$.  On the other hand, $G/G_n\twoheadrightarrow(G/L)/(G_nL/L)$.  Therefore, $(G/L)/(G_nL/L)$ is finite-by-abelian and so is in $W$.  It now follows from Lemma \ref{lem:ext_xi} that $G/L\in W$.
\end{proof}

Combining Lemmas \ref{lem:xi_indp}, \ref{lem:finandab}, \ref{lem:union_xi}, \ref{lem:ext_xi}, and \ref{lem:xi_quot}, we obtain the following corollary.

\begin{cor}
If $G\in \EA$, then $G\in\mathrm{W}$.
\end{cor}

We thus produce a characterization of elementary amenable groups.

\begin{thm}\label{thm:EA_char}
Let $G$ be a marked group. Then the following are equivalent:
\begin{enumerate}[(1)]
\item $G\in\EA$ .
\item $T^l(G)$ is well-founded for all $l\geq 1$.
\item $T^l(G)$ is well-founded for some $l\geq 1$.
\end{enumerate}
\end{thm}

We can rephrase this to have the form of a chain condition independent of the marking. This corollary may thus be taken to be a definition of elementary amenability.

\begin{cor}
A countable group $G$ is elementary amenable if and only if there is no infinite descending sequence of the form 
\[
G=G_0\geq G_1\geq\ldots
\]
such that for all $n\geq 0$, $G_n\neq\{e\}$ and there is a finitely generated subgroup $K_n\leq G_n$ with $G_{n+1}= [K_n,K_n]\cap H_n$ where $H_n$ is the intersection of the index-$(\leq(n+1))$ normal subgroups of $K_n$.
\end{cor}

\begin{proof}
Suppose $G\in\Gw$ and there is an infinite descending sequence
\[
G=G_0\geq G_1\geq\ldots
\]
as in the statement. Form $T^1(G)$, the decomposition tree of $G$ with offset $1$. We now proceed by induction to build $s_0\subsetneq s_1\subsetneq \dots$ with $s_i\in T^1(G)$ and $|s_i|=i$ such that $G_i\hookrightarrow G_{s_i}$. The base case is immediate: set $s_0=\emptyset$. Suppose we have defined $s_n$, so $G_n\hookrightarrow G_{s_n}$. Let $K_n\leq G_{n}$ be such that $G_{n+1}= [K_n,K_n]\cap H_n$ where $H_n$ is the intersection of the index-$(\leq n+1)$ normal subgroups of $K_n$. Since $K_n$ is finitely generated, there is $R_{m}(G_{s_n})$ such that $K_n\hookrightarrow R_{m}(G_{s_n})$. It follows that $G_{n+1}\hookrightarrow G_{s_n\conc m}$. Setting $s_{n+1}=s_n\conc m$, we have verified the inductive claim.  The tree $T^1(G)$ thus has an infinite branch, so by Theorem~\rm\ref{thm:EA_char}, $G\notin\EA$.

\medskip

Suppose there are no infinite descending sequences as in the statement and form $T^1(G)$. Let $s_0\subsetneq s_1\subsetneq \dots$ with $s_i\in T^1(G)$ and $|s_i|=i$. It suffices to show $s_0\subsetneq s_1\subsetneq \dots$ terminates, and so $T^1(G)$ is well-founded. This is indeed obvious since by construction the sequence of subgroups $G_{s_0}\geq G_{s_1}\geq\dots$ is a sequence of subgroups as in the chain condition.
\end{proof}

There are two main differences between this chain condition and the chain conditions explored in the earlier sections of this paper.  First of all, $G_{n+1}$ is not related to $G_n$ only by being a subgroup.  This is not unheard of; for example when looking at weak chain conditions one requires that $G_{n+1}$ be an \emph{infinite index} subgroup of $G_n$.  The second difference is that the definition of $H_n$ changes with $n$.  As far as we are aware, there are no widely-studied chain conditions defined in this way.  That elementary amenability can be recast this way suggests that perhaps there are other interesting chain conditions with this property.

\subsection{$\EA$ is not Borel}
We now study the descriptive-set-theoretic properties of $\EA$. We show that on $\EA$ the decomposition rank is unbounded below $\omega_1$. 

\begin{lem}\label{lem:xiRankSucc}
For every $K\in \EA$, there is $L\in \EA$ with $\xi(K)<\xi(L)$.
\end{lem}

\begin{proof}
Let $G\in \EA$ be as given by Proposition~\rm\ref{prop:EAembedding} for $K$ and form $L:=G\wr \Zb$.  Let $k=\deg(L)$, and take $L_i$ to be the subgroup of $L$ corresponding to $i\in T^k(L)$. Since $L$ is finitely generated, we may find $n$ such that $L=R_n(L)$. \par

\indent We now consider $L_n$. The group $[L,L]=[R_n(L),R_n(L)]$ certainly contains $[M,M]=[H,H]^{<\Zb}$. On the other hand, if $N\normal L$ has index $k+1$, there is $n\in \Zb\setminus \{0\}$ so that $n\in N$. Applying Lemma~\ref{lem:wreath}, $[G,G]\leq N$, so $[H,H]^{<\Zb}\leq L_n$. The group $K$ thus embeds into $[H,H]$, and Lemma~\rm\ref{lem:sgrp_xi} implies  $\xi(K)\leq \xi(L_n)$. Appealing to Lemma~\rm\ref{lem:rk_xi}, we conclude $\xi(K)< \xi(L)$ proving the lemma.
\end{proof}

Our next lemma follows immediately from Corollary~\rm\ref{lem:sgrp_xi} by taking the direct sum.
\begin{lem}\label{lem:xiRankLimit}
Let $\{A_i\}_{i\in \Nb}$ be countable groups. If $A_i\in \EA$ for all $i\in \Nb$, then there is $A\in \EA$ with $\xi(A)\geq \xi(A_i)$ for all $i\in \Nb$.
\end{lem}

\begin{lem}
For all $\beta<\omega_1$, there is $G\in \EA$ such that $\xi(G)\geq \beta$.
\end{lem}
\begin{proof}
The proof is the same as that of Lemma \ref{lem:CentRankUnbdd}, with Lemmas \ref{lem:xiRankSucc} and \ref{lem:xiRankLimit} referenced at the appropriate places.
\end{proof}

\begin{lem}\label{lem:xi_unbounded}
For each $\beta<\omega_1$, there is a finitely generated $G\in \EA$ such that $\xi(G)\geq \beta$.
\end{lem}

\begin{proof}
Let $H\in\mc \EA$ be a group such that $\xi(H)\geq\beta$. Proposition~\rm\ref{prop:EAembedding} implies that $H$ embeds into a 2-generated group $G\in\mc \EA$.  By Corollary~\ref{lem:sgrp_xi}, $\xi(G)\geq\xi(H)\geq\beta$.
\end{proof}

\begin{thm}\label{thm:EA_nonborel}
$\EA$ is a non-Borel $\Pi^1_1$ set in $\ms{G}$, and $\EA \cap \Gw_{fg}$ is a non-Borel $\Pi^1_1$ set in $\ms{G}_{fg}$.
\end{thm}
\begin{proof}
This follows from Theorem \ref{thm:BddnessThm}, Lemma \ref{lem:Phi_borel}, and Lemma~\rm\ref{lem:xi_unbounded} along with the facts that $\xi(G)\leq \rho(T^1(G))$ and that $\rho\circ \Phi^1$ is a $\Pi^1_1$-rank on $\EA$.
\end{proof}

Let $\AG\subseteq\Gw$ denote the class of countable amenable groups. Via Theorem~\rm \ref{thm:EA_nonborel}, we now may give a non-constructive answer to an old question of Day \cite{D57}, which was open until Grigorchuk \cite{G84} constructed groups of intermediate growth: \textit{Is it the case that every amenable group is elementary amenable}?

\begin{cor}
There is a finitely generated amenable group that is not elementary amenable.
\end{cor}
\begin{proof}
It is well-known that $\AG$ is Borel; see Lemma~\rm\ref{lem:AG_borel} for a proof.  The set $\AG\cap\Gw_{fg}$ is thus Borel.  On the other hand, Theorem~\rm \ref{thm:EA_nonborel} gives that $\EA\cap \Gw_{fg}$ is not Borel. We conclude that $\EA\cap \Gw_{fg} \subsetneq \AG\cap \Gw_{fg}$.
\end{proof}

\subsection{Further observations}

By a result of C. Chou \cite[Proposition 2.2.]{Ch80}, the class of elementary amenable groups is the smallest class of countable discrete groups that satisfies (i),(ii), and (iii) of Definition~\rm\ref{def:EA}. Chou's theorem suggests a natural ranking of elementary amenable groups different than our decomposition rank. Indeed, after \cite{Os02}, define
\begin{enumerate}[$\bullet$]
\item $G\in \EA_0$ if and only if $G$ is finite or abelian.

\item Suppose $\EA_{\alpha}$ is defined. Put $G\in \EA_{\alpha}^e$ if and only if there exists $N\trianglelefteq G$ such that $N\in \EA_{\alpha}$ and $G/N\in \EA_{0}$. Put $G\in \EA_{\alpha}^l$ if and only if $G=\bigcup_{i\in \Nb}H_i$ where $(H_i)_{i\in \Nb}$ is an $\subseteq$-increasing sequence of subgroups of $G$ with $H_i\in \EA_{\alpha}$ for each $i\in \Nb$. Set $\EA_{\alpha+1}:=\EA_{\alpha}^e\cup \EA_{\alpha}^l$.

\item For $\lambda$ a limit ordinal, $\EA_{\lambda}:=\bigcup_{\beta<\lambda}\EA_{\beta}$.
\end{enumerate}

By a result of D. Osin \cite[Lemma 3.2]{Os02}, $\bigcup_{\alpha <\omega_1}\EA_{\alpha}$ is closed under group extension. It now follows from Chou's theorem that $\EA=\bigcup_{\alpha <\omega_1}\EA_{\alpha}$. One may then define for $G\in \EA$
\[
\rk(G):=\min\{\alpha\;|\;G\in \EA_{\alpha}\}.
\]
\noindent We call $\rk(G)$ the \textbf{construction rank} of $G$.

We now compare $\xi$ and $\rk$ and in the process mostly recover a theorem of Olshanskii and Osin.

\begin{prop}\label{prop:rk_xi}
For $G\in \EA$, $\rk(G)\leq 3\xi(G)$. 
\end{prop}

\begin{proof} 
We induct on $\xi(G)$ for the proposition. For the base case, if $\xi(G)=1$, then $G=\{e\}$, and the inductive claim obviously holds. Suppose the proposition holds up to $\beta$. Say $\xi(G)=\beta+1$ and $\deg(G)=k$.  Then $\xi(G_i)\leq \beta$ for each $G_i$ associated to $i\in T^k(G)$, and applying the inductive hypothesis, $\rk(G_i)\leq 3\xi(G_i)$.\par 

\indent On the other hand, $R_i(G)/G_i$ is finite-by-abelian, say an extension of the group $A$ by the group $F$.  Letting $F_0$ be the inverse image of $F$ in $R_i(G)$ under the usual projection, $\rk(F_0)\leq\rk(G_i)+1$, and $R_i(G)/F_0\cong A$.  Hence,
\[
\rk(R_i(G))\leq  (\rk(G_i)+1)+1 \leq 3\xi(G_i)+2.
\]
We conclude
\[
\rk(G)\leq \sup_{i\in \Nb}\left(3\xi(G_i)+2 \right)+1\leq 3(\beta+1)=3\xi(G).
\]
This finishes the induction.
\end{proof}

Bounding $\xi$ from above by $\rk$ involves a bit more work. We begin with a general lemma for well-founded trees.

\begin{lem}\label{lem:wf_tree}
Suppose $T$ is a well-founded tree. Then
\[
\rho(T)\leq \sup_{|s|=k}\rho(T_s) + k.
\]
\end{lem}
\begin{proof}
We argue by induction on $|s|$. For the base case, $|s|=1$, 
\[
\rho(T)=\rho_T(\emptyset)+1=\sup_{i\in T}\left(\rho_T(i)+1\right)+1=\sup_{i\in \Nb}\rho(T_i)+1.
\]

\indent Supposing the lemma holds up to length $k$,
\[
\rho(T)\leq \sup_{|s|=k}\rho(T_s)+k\leq \sup_{|s|=k}\left(\sup_{s\conc i\in T}\rho(T_{s\conc i})+1\right)+k\leq\sup_{|s|=k+1}\rho(T_s)+k+1
\]
completing the induction.
\end{proof}

\begin{prop}\label{prop:xi_rk_ub}
For $G\in \EA$, 
\[
\rho\left(T^1(G)\right)\leq \omega(\rk(G)+1).
\]
In particular, $\xi(G)\leq \omega(\rk(G)+1)$.
\end{prop}

\begin{proof}
We argue by induction on $\rk(G)$. For the base case, $\rk(G)=0 $, $G$ is either finite or abelian. There is thus $m\geq 1$ such that every element of $T^1(G)$ has length at most $m$. It follows that $\rho(T^1(G))$ is finite, which proves the base case.\par

\indent Suppose the lemma holds up to $\alpha$ and $\rk(G)=\alpha+1$. Let us consider first the case that the construction rank is given by a countable increasing union; say $G=\bigcup_{n\in \omega}H_n$ with $\rk(H_n)\leq \alpha$ for each $n$. Since $R_i(G)$ is finitely generated, there is $n(i)$ for which $G_i\leq H_{n(i)}$. We apply the inductive hypothesis and Lemma~\rm\ref{lem:xi_indp} to conclude
\[
\rho\left(T^2(G_i)\right)\leq \rho(T^1(G_i))\leq \omega(\alpha+1).
\]
Hence,
\[
\rho\left(T^1(G)\right)=\sup_{i\in \omega}\rho\left(T^2(G_i)\right) +1 \leq \omega\cdot \alpha +\omega+1\leq \omega(\alpha+2),
\]
verifying the hypothesis in this case. \par

\indent We now consider the case $\rk(G)$ is given by a group extension. Suppose $H\trianglelefteq G$ is such that $\rk(H)=\alpha$ and $\rk(G/H)=0$. If $G/H$ is abelian, $G_i\leq H$ for each $i$. Hence, $\rk(G_i)\leq \alpha$, and the desired result follows just as in the increasing union case. Suppose $G/H$ is finite. We may find $k$ such that for all $s\in T^1(G)$ with $|s|=k$, $G_s\leq H$. Applying the inductive hypothesis and Lemma~\rm\ref{lem:xi_indp},
\[
\rho\left(T^{k+1}(G_s)\right)\leq \rho\left(T^{1}(G_s)\right)\leq \omega(\alpha+1).
\]
Lemma~\rm\ref{lem:wf_tree} now implies
\[
\rho\left(T^1(G)\right)\leq \sup_{|s|=k}\rho\left(T^1(G)_s\right)+k\leq \omega(\alpha+1)+k\leq \omega(\alpha+2).
\]
This completes the induction, and we conclude the proposition.
\end{proof}

As a corollary to Lemma~\rm\ref{lem:xi_unbounded} and Proposition~\rm\ref{prop:xi_rk_ub}, we obtain a less detailed version of a theorem from the literature.

\begin{cor}[Olshanskii, Osin {\cite[Corollary 1.6]{OO13}}]\label{cor:OlOs}
For every ordinal $\alpha<\omega_1$, there is $G\in \EA\cap \Gw_{fg}$ such that $\alpha\leq\rk(G)$. The function $\rk\from\EA\cap \Gw_{fg}\to ORD$ is thus unbounded below $\omega_1$.
\end{cor}

\begin{proof} Suppose for contradiction $\alpha <\omega_1$ is such that $\rk(G)< \alpha$ for all $G\in \EA$. By Proposition~\rm\ref{prop:xi_rk_ub}, $\xi(G)\leq \omega(\alpha+1)<\omega_1$ for all $G\in\EA$, contradicting Lemma~\rm\ref{lem:xi_unbounded}. 
\end{proof}

In our proof of Theorem~\rm\ref{thm:EA_nonborel}, we use that $\rho\circ\Phi^1$ is a $\Pi^1_1$-rank. It is natural to ask if $\xi$ itself is a $\Pi^1_1$-rank. This is indeed the case.
\begin{thm}
The decomposition rank is a $\Pi^1_1$-rank on $\EA$.
\end{thm}

\begin{proof}
Each of the ranks $\phi_l:=\rho\circ \Phi^l$ is a $\Pi^1_1$-rank on $\EA$ where $\Phi^l$ is as defined in Lemma~\rm\ref{lem:Phi_borel}. Let $\leq_l^{\Pi}\subseteq \Gw\times \Gw$ and $\leq_l^{\Sigma} \subseteq \Gw\times \Gw$ be the relations given by $\phi_l$ as a $\Pi^1_1$-rank. We now consider the following relations:
\[
\leq_{\xi}^{\Pi}:=\bigcup_{N\in \Nb}\bigcap _{l\geq N}\leq_l^{\Pi}\text{ and } \leq_{\xi}^{\Sigma}:=\bigcup_{N\in \Nb}\bigcap _{l\geq N}\leq_l^{\Sigma}
\]
Since co-analytic and analytic sets are closed under countable unions and intersections, $\leq_{\xi}^{\Pi}$ is co-analytic and $\leq_{\xi}^{\Sigma}$ is analytic. To conclude $\xi$ is a $\Pi^1_1$-rank, it thus remains to show for $H\in \EA$, 
\begin{align*}
G\in \EA \,\wedge\, \xi(G)\leq \xi(H) &\Leftrightarrow G \leq_{\xi}^\Sigma H \\
 &\Leftrightarrow G \leq_\xi^\Pi H.
\end{align*}

Suppose $G\in \EA$ and $\xi(G)\leq \xi(H)$. Letting $M:=\max\{\deg(G),\deg(H)\}$, we see that $\rho\left(T^k(G)  \right)\leq \rho \left(T^k(H)  \right)$ for all $k\geq M$ via Lemma~\rm\ref{lem:xi_indp}, hence $\phi_k(G)\leq \phi_k(H)$ for $k\geq  M$. We conclude that $G \leq_{\xi}^{\Pi} H$ and $G\leq_{\xi}^{\Sigma} H$.\par

\indent Conversely, suppose $G \leq_{\xi}^{\Pi} H$ and $G\leq_{\xi}^{\Sigma} H$ and let $M\geq 0$ be such that $G \leq_{k}^{\Pi} H$ and $G \leq_{k}^{\Sigma} H$ for all $k\geq M$. Immediately, $G\in \EA$. For each $k\geq M$, we further see $\phi_k(G)\leq \phi_k(H)$, and taking $k=\max\{\deg(G),\deg(H),M\}$, 
\[
\xi(G)=\phi_k(G)\leq \phi_k(H)=\xi(H).
\]
Therefore, $\xi$ is a $\Pi^1_1$-rank.
\end{proof}

Propositions \ref{prop:rk_xi} and \ref{prop:xi_rk_ub} combine to give us
$$ \xi(G) \leq \omega(\rk(G)+1) \leq \omega(3\xi(G)+1), $$
so $\rk$ is closely related to a $\Pi^1_1$-rank.  Given this close relationship, it is natural to ask whether or not $\rk$ is a $\Pi^1_1$-rank. We suspect, however, that $\rk$ is \textit{not} a $\Pi_1^1$-rank as the sets $\rk^{-1}(\alpha)$ are likely analytic and non-Borel for suitably large $\alpha$; in fact, we believe $\rk^{-1}(2)$ is analytic and non-Borel. Indeed, if $\EA_\alpha$ is Borel and uncountable, then $\EA^e_{\alpha}$ is defined by quantifying over $\EA_\alpha$. We thus expect $\EA_{\alpha}^e$ to be analytic and, barring some clever argument, non-Borel.  (We remark that one can make such a clever argument in the case of $\EA^e_0$, but it does not seem to work beyond that.) We do not pursue this question further as it is tangential to the aim of this work and somewhat technical.

\section{Borel functions and sets}\label{sec:Borel}

In previous sections we made claims that certain maps and sets were Borel, and from this and the Boundedness Theorem \ref{thm:BddnessThm}, we concluded that certain subsets of $\Gw$ were not Borel. A slogan from descriptive set theory is ``Borel = explicit'' meaning if you describe a map or set without an appeal to something like the axiom of choice or quantifying over an uncountable space, it should be Borel.  As the maps and sets from previous sections are ``explicit'' in this sense, we were content to state they were Borel without further proof.  To those not as familiar with descriptive set theory, we offer this section to verify our previous claims. 

Recall that $\Gw=\{ N\normal\Fw \}$ and that we identify $N$ with the group $\Fw/N$. We make frequent use of the usual projection from $\Fw$ to $\Fw/N$ and always denote this projection by $f_N$.  Every countable group is identified with an element of $\Gw$; in fact, a given group $G$ corresponds to many distinct elements of $\Gw$ as there are many different surjections of $\Fw$ onto $G$. We fix an enumeration $(\gamma_i)_{i\in \Nb}$ for $\Fw$, and this gives rise to an enumeration of $G$ in the obvious way.  Let us also enumerate the generators for $\Fw$ as $(a_i)_{i\in \Nb}$. Recall finally that $\mc G_{fg} = \cup_{n\in\Nb} \{ N \normal \Fw \mid \forall k\geq n \;a_k\in N \}$.  This is an $F_\sigma$ subset of $\Gw$.  In particular, its Borel sets as a Borel space are precisely those sets of the form $B\cap\Gw_{fg}$ where $B\subseteq\Gw$ is Borel.

\subsection{Borel functions}
The sub-basic open sets of $\Gw$ are those of the form $O_{\gamma}=\{ N \mid \gamma \in N\}$ and their complements.  The Borel $\sigma$-algebra on $\Gw$ is thus generated by the $O_\gamma$, so in order to show a map $\psi \from\Gw\to\Gw$ is Borel, we need only to check that $\psi^{-1}(O_\gamma)$ is Borel for all $\gamma\in\Fw$.\par

\indent We begin with the easier examples of Borel maps. 

\begin{lem}\label{lem:QuotientBorel}
For each $\delta\in\Fw$, there is a Borel map $Q_\delta\from\Gw\to\Gw$ such that if $N\in\Gw$ with $\Fw/N\cong G$, then $\Fw/Q_\delta(N)\cong G/\ngrp{f_N(\delta)}$.
\end{lem}
\begin{proof}
Since $G/\ngrp{f_N(\delta)}\cong \Fw/\ngrp{N,\delta}$, the map $Q_\delta(N):=\ngrp{\delta}N$ meets our requirements.  We need only check that it is Borel. For this,
\begin{align*}
Q_\delta^{-1}(O_\gamma) &= \{ N\in\Gw \mid \gamma\in\ngrp{\delta}N \} \\
 &= \{ N\in\Gw \mid \exists g\in\ngrp{\delta} \; g^{-1}\gamma\in N \} \\
 &= \bigcup_{g\in\ngrp{\delta}} \{ N\in\Gw \mid g^{-1}\gamma\in N\}
\end{align*}
which is open, so we have verified the lemma.
\end{proof}

We can now easily prove Lemma \ref{lem:MaxNMapBorel}.

\begin{proof}[Proof of Lemma \ref{lem:MaxNMapBorel}]
By repeated composition, we may define $Q_s\from\Gw\to\Gw$ for all $s\in\wbaire\setminus \{\emptyset\}$ so that 
\[
Q_s(N)=\ngrp{\gamma_s}N;
\]
we define $Q_{\emptyset}:=id$. The previous lemma ensures these maps are Borel.\par

\indent Now suppose $t\in\wbaire$ is of the form $v\conc i$ with $v\in\wbaire$ and $i\in\Nb$ and consider the basic open set $O_t:=\{T\in Tr\mid t\in T\}$ of $Tr$. We see that
\[
\Phi_{M_n}^{-1}(O_t)=\{ N\in\Gw \mid Q_v(N)\neq Q_t(N)\},
\]
which is Borel. The map $\Phi_{M_n}$ is thus Borel.
\end{proof}

\begin{lem}\label{lem:Rn_borel}
For each $n\geq 0$, there is a Borel map $R_n\from\Gw\to\Gw$ such that if $N\in\Gw$ with $\Fw/N\cong G$, then $\Fw/R_n(N)\cong\<g_0,\ldots,g_n\>$.
\end{lem}

\begin{proof}
Let $\pi_n\from\Fw\to\Fw$ be induced by mapping the generators $(a_i)_{i\in \Nb}$ as follows:
\[\pi_n(a_i) = 
\begin{cases}
\gamma_i, & 0\leq i\leq n \\
e, & \text{otherwise.}
\end{cases}
\]
Suppose that $N\in\Gw$ with $\Fw/N=G$.  The function $f_N\circ\pi_n\from\Fw\to \grp{g_0,\dots,g_n}$ is then a surjection. We thus define $R_n$ to be the map sending $N$ to $\ker (f_N\circ\pi_n)$. Since $\Fw/\ker (f_N\circ\pi_n)\cong \grp{g_0,\dots,g_n}$, this works as intended. \par

\indent As $\gamma\in\ker (f_N\circ\pi_n)$ iff $\pi_n(\gamma)\in\ker (f_N)$, we conclude that
\[
R_n^{-1}(O_\gamma) = \{ M \in\Gw \mid \pi_n(\gamma)\in M \},
\]
which is the open set $O_{\pi_n(\gamma)}$. The map $R_n$ is thus Borel.
\end{proof}

The above proof works for subgroups generated by any fixed collection of elements of $G$; that is to say, the same proof shows the maps $G\mapsto G_s$ defined in Section \ref{sec:Max} are Borel, so as before, we get Lemma \ref{lem:MaxMapBorel} as a corollary.

We now move on to proving Lemma \ref{lem:CentMapBorel}; this follows from the next lemma. Its proof is more involved than the previous two.

\begin{lem}
For each $s\in\wbaire\setminus \{\emptyset\}$, there is a Borel map $C_s\from\Gw\to\Gw$ such that if $N\in\Gw$ with $\Fw/N\cong G$, then $\Fw/C_s(N)\cong C_G(\{g_s\})$.
\end{lem}

\begin{proof}
Suppose that $N\in\Gw$ and $\Fw/N\cong G$.  Define $\pi_N\from\Fw\to\Fw$ by
\[
\pi_N(a_j):=
\begin{cases}
\gamma_j, & \text{ if }f_N(\gamma_j)\in C_G\left(\{f_N(\gamma_s)\}\right) \\
e, & \text{ else.}
\end{cases}
\]
The map $f_N\circ \pi_N\from\Fw\to C_G(\{g_s\})$ is then a surjection, so the map $N\mapsto\ker(f_N\circ\pi_N)$ works as intended.  In order to check it is Borel, we introduce the set 
\[
S_j:=\{N\in\Gw \mid \pi_N(a_j)=\gamma_j\}.
\]
Since $f_N(\gamma_j)\in C_G(\{f_N(\gamma_s)\})$ iff $[\gamma_j,\gamma_{s_i}]\in N$ for each $0\leq i \leq |s|-1$, 
\[
S_j=\{N\in\Gw \mid [\gamma_j,\gamma_{s_i}]\in N \text{ for each }0\leq i \leq |s|\},
\] 
which is an open set.

We now fix a word $\delta=\delta(a_0,\dots,a_m)\in \Fw$ and consider the pre-image of the basic open set $O_{\delta}$. Our notation $\delta(a_0,\dots,a_m)$ indicates the word $\delta$ only uses the letters appearing in the parentheses.  We may evaluate $\pi_N(\delta)$ by substituting in the images of $a_0,\dots,a_m$, so $\pi_{N}(\delta)=\delta(x_0,\dots,x_m)$ for some $\ol{x}:=(x_0,\dots,x_m)\in \Omega:=\prod_{i=0}^m\{\gamma_i,e\}$ that depends on $N$.  The set of $N\in \Gw$ such that $\pi_N(\delta(a_0,\ldots,a_m))=\delta(\ol{x})$ for some fixed $\ol{x}$ is the Borel set 
\[
S_{\ol{x}}:=\bigcap_{x_j=\gamma_j} S_j \cap \bigcap_{x_k=e} S_k^c.
\] 
Since $\delta\in\ker(f_N\circ \pi_N)$ iff $\pi_N(\delta)\in\ker f_N=N$, we now see that
\begin{align*}
C_s^{-1}(O_\delta) &= \{N\in\Gw \mid \pi_N(\delta)\in N\} \\
 &= \bigcup_{\ol{x}\in\Omega} \left( \{N\in\Gw \mid \delta(\ol{x})\in N\} \cap S_{\ol{x}} \right)
\end{align*}
which is Borel.
\end{proof}

We next show the maps $S_k$ from Section \ref{sec:EAGroups} are Borel. The main idea is the same as in previous lemma.

\begin{lem}\label{lem:Sk_Borel} 
For each $k\geq 1$, there is a Borel map $S_k:\Gw_{fg}\rightarrow \Gw$ such that if $\Fw/N=G$, then
\[
\Fw/S_k(N)\cong [G,G]\cap \bigcap \mc{N}_k(G)
\]
where $\mc{N}_k(G):=\{M\trianglelefteq G\mid |G:M|\leq k+1\}$.
\end{lem}
\begin{proof}
Suppose $N\in \Gw_{fg}$ and $G\cong\Fw/N$.  Similarly to the previous lemma, we define $\pi_N\from\Fw\to\Fw$ by 
\[
\pi_N(a_i):=
\begin{cases}
\gamma_i, & \text{ if }f_N(\gamma_i)\in [G,G]\cap\bigcap\mc{N}_k(G)\\
e, & \text{ else.}
\end{cases}
\]
Define $S_k:\Gw_{fg}\rightarrow \Gw$ by $N\mapsto \ker(f_N\circ\pi_{N})$; this map behaves as desired.  We claim this map is also Borel. \par

\indent Define
\[
\mc N_k := \left\{ M\in\Gw_{fg} \mid |\Fw : M|\leq k+1 \right\}.
\]
If $N\in\Gw_{fg}$, then the collection of index-$\leq k+1$ subgroups of $\Fw/N$ is precisely $\{ MN/N \mid M\in\mc N_k \}$.  Therefore, $f_N(\gamma_i)\in [G,G]\cap\bigcap\mc N_k(G)$ iff $\gamma_i\in [\Fw,\Fw]N\cap \bigcap_{M\in\mc N_k} MN $.  As in the previous lemma, we may define
\begin{align*}
S_i &:= \left\{N\in \Gw_{fg} \mid \pi_N(a_i)=\gamma_i  \right\}\\
	&= \left\{N\in \Gw_{fg} \mid \gamma_i \in [\Fw,\Fw]N\cap  \bigcap_{M\in\mc N_k} MN \right\}\\
	&= \bigcup_{\delta\in[\Fw,\Fw]} \left\{N\in\Gw_{fg} \mid \delta^{-1}\gamma_i\in N \right\} \cap \bigcap_{M\in \mc{N}_k}\bigcup_{\delta\in M}\left\{N\in\Gw_{fg} \mid \delta^{-1}\gamma_i\in N\right\}.
\end{align*}
The last set is Borel since $\mc{N}_k$ is countable.  Given $\ol{x}:=(x_0,\dots,x_m)\in \Omega:=\prod_{i=0}^m\{\gamma_i,e\}$, we define as before $S_{\ol{x}}$.  \par

\indent Fixing a word $\delta=\delta(a_0,\dots,a_m)\in \Fw$, we now consider the pre-image of the basic open set $O_{\delta}$. We see
\begin{align*}
S_k^{-1}(O_\delta) &= \{N\in\Gw \mid \pi_N(\delta)\in N\} \\
 &= \bigcup_{\ol{x}\in\Omega} \big( \{N\in\Gw \mid \delta(\ol{x})\in N\} \cap S_{\ol{x}} \big)
\end{align*}
which is Borel.
\end{proof}

\indent Using Lemma~\ref{lem:Rn_borel} and \ref{lem:Sk_Borel}, we build Borel maps $\Psi^l_s:\ms{G}\rightarrow \ms{G}$ for each $l\in\Nb$ and $s\in \wbaire$. For $s=\emptyset$, put $\Psi^l_{\emptyset}=id$. Supposing we have defined $\Psi^l_s$, define $\Psi^l_{s\conc n}$ by 
\[
\Psi^l_{s\conc n}(N):=S_{|s|+l}\circ R_n(\Psi^l_s(N)).
\]
It follows that if $s\in T^l(G)$ with $G=\Fw/N$, then $\Fw/\Psi^l_s(N)= G_s$. If $s\notin T^l(G)$, then $\Fw/\Psi^l_s(N)=\{e\}$.
 
\begin{proof}[Proof of Lemma \ref{lem:Phi_borel}]
\indent Fixing $s\in \wbaire$ and $l\in\Nb$,
\[
(\Phi^l)^{-1}(O_s)=\left\{N\in \ms{G} \mid s\in T^l(\Fw/N)\right\}.
\]
If $s=\emptyset$, then $(\Phi^l)^{-1}(O_s)=\ms{G}$ which is plainly Borel. Else, say $s=r\conc n$, so
\[
\begin{array}{ccl}
(\Phi^l)^{-1}(O_s) & = & \left\{N\in \ms{G} \mid r\conc n \in T^l(\Fw/N)\right\}\\
				& = & \left\{N\in \ms{G} \mid (\Fw/N)_r\neq \{e\}\right\}\\
				& = & (\Psi^l_r)^{-1}(\ms{G}\setminus \{e\}),
\end{array}
\]
which is Borel.
\end{proof}

\subsection{Borel sets}
Recall that $\AG$ denotes the class of countable amenable groups.

\begin{lem}[Folklore]\label{lem:AG_borel}
The set $\AG$ is Borel in $\ms{G}$, and therefore, $\AG\cap \Gw_{fg}$ is Borel.
\end{lem}
\begin{proof}
Amenable groups are characterized by F\o lner's property: A countable group $G$ is amenable if and only if for every finite $F\subseteq G$ and every $n\geq 1$, there is a finite non-empty subset $K\subseteq G$ such that 
\[
\frac{|xK\Delta K|}{|K|}\leq \frac{1}{n}
\]
for all $x\in F$ where $\Delta$ denotes the symmetric difference.\par

\indent Letting $P_f(\Fw)$ be the collection of finite subsets of $\Fw$, we infer
\[
\AG=\bigcap_{F\in P_f(\Fw)}\bigcap_{n\geq 1}\bigcup_{K\in P_f(\Fw)}\bigcap_{x\in F}\left\{N\in \Gw \mid \frac{|f_N(x)f_N(K)\Delta f_N(K)|}{|f_N(K)|}\leq \frac{1}{n}\right\}.
\]
It thus suffices to show
\[
\Omega:=\left\{N\in \Gw \mid \frac{|f_N(x)f_N(K)\Delta f_N(K)|}{|f_N(K)|}\leq \frac{1}{n}\right\}
\]
is Borel. It is easy to see requiring $|f_N(K)|=m$ and $|f_N(x)f_N(K)\Delta f_N(K)|=l$ is Borel, hence
\[
\Omega=\bigcup_{\frac{l}{m}\leq \frac{1}{n}}\left\{N\mid |f_N(x)f_N(K)\Delta f_N(K)|=l\text{ and } |f_N(K)|=m \right\}
\]
is Borel. The set $\AG$ is thus Borel.
\end{proof}

\section{Further remarks}\label{sec:Remarks}

Our results give tools to study groups enjoying any of the other chain conditions in the literature. Perhaps more interestingly, our results suggest new questions concerning elementary amenable groups and groups with the minimal condition on centralizers, maximal condition on subgroups, and maximal condition on normal subgroups. \par

\indent Most immediately, one desires a better understanding of the various rank functions. In the case of max groups, there are no infinite subgroup rank two groups, the infinite groups with subgroup rank 3 are Tarski monsters, and $\Zb$ has rank $\omega+1$. In the case of max-n, examples of finite rank groups are easy to produce and understand; however, transfinite rank examples are somewhat mysterious. Following Olshanskii and Osin, cf. \cite[Corollary 1.6]{OO13}, we ask
\begin{quest} 
For which ordinals $\alpha$ is there an infinite group in $\mc M_C$ ($\mc M_{\max{}}, \mc M_n$) such that the centralizer rank (subgroup rank, length) is $\alpha$?
\end{quest}

\indent In a different direction, showing a set is non-Borel in $\Gw$ demonstrates there is no ``simple" definition of the class. Our techniques give a way to determine if a subset of $\Gw$ (or of a Borel subset of $\Gw$) given by a chain condition is not Borel and hence to determine if it does not admit a ``simple" characterization. In the setting of max-n groups, there is a particularly intriguing question along these lines.  By an old result of Hall, a two-step solvable group is max-n if and only if it is finitely generated; this is certainly a Borel condition. On the other hand, no such nice characterization of three-step solvable groups with max-n is known. We thus ask

\begin{quest} 
Is the set of max-n three-step solvable marked groups Borel? 
\end{quest}

\indent In a similar vein, our results on elementary amenable groups, in a sense, show elementary amenable groups are not ``elementary". One naturally asks

\begin{quest}[Hume] Is there an intermediate ``elementary" Borel set between $\EA\cap \Gw_{fg}$ and $\AG\cap \Gw_{fg}$? More precisely, is there an elementary class $\ms{E}(B)$ in the sense of Osin \cite{Os02} with $B$ ``small" such that $\EA\cap \Gw_{fg}\subseteq \ms{E}(B)\subsetneq \AG\cap\Gw_{fg}$ and $\ms{E}(B)$ is Borel? 
\end{quest}

We also arrive at new questions with a descriptive-set-theoretic flavor.
 
\begin{defn}
Let $Y$ be a uncountable Polish space.  A set $A\subseteq Y$ is \textbf{$\Pi^1_1$-complete} if $A$ is $\Pi^1_1$ and for all $B\subseteq X$ with $X$ an uncountable Polish space and $B$ co-analytic, $B$ Borel reduces to $A$.
\end{defn}

The idea is that $\Pi^1_1$-complete sets are as complicated as they possibly could be; Theorem \ref{thm:WFComplete} says that $WF\subseteq Tr$ is $\Pi^1_1$-complete. 

\begin{quest}
Are any of $\mc M_C,\mc M_{\max{}},\mc M_n,$ or $\EA$ $\Pi^1_1$-complete?
\end{quest}

Note that for a positive answer it suffices to show that $WF$ (or some other $\Pi^1_1$-complete set) Borel reduces to these sets.  Under an extra set-theoretic assumption known as $\Sigma^1_1$-Determinacy, every $\Pi^1_1$ set which is not Borel is in fact $\Pi^1_1$-complete.  We do not expect that extra set-theoretic assumptions should be necessary to prove any of the sets are $\Pi^1_1$-complete; we mention this as evidence that the positive answer is indeed the correct one. It is worth noting the question is a problem in group theory. For example, in the case of $\EA$ one must devise a method of building a group from a tree so that well-founded trees give rise to elementary amenable groups and ill-founded trees give rise to non-elementary-amenable groups.

\subsection*{Acknowledgments}
The authors would like to thank Alexander Kechris and Andrew Marks for helpful mathematical discussions.

J. Williams was partially supported by NSF Grant 1044448, Collaborative Research: EMSW21-RTG: Logic in Southern California.


\bibliographystyle{bibgen}
\bibliography{biblio}

\end{document}